\documentclass[a4paper,12pt]{amsart}
\usepackage{mathrsfs}
\usepackage{amsfonts,amssymb,amscd,amsthm,amsmath,graphicx}
\usepackage{longtable}
\usepackage[all]{xy}

\theoremstyle{plain}
\newtheorem{theorem}{Theorem}[section]

\newtheorem{lemma}{Lemma}[section]
\newtheorem{proposition}{Proposition}[section]
\newtheorem{definition}{Definition}[section]
\newtheorem{question}{Question}[section]

\newtheorem{example}{Example}[section]

\def\A{\operatorname{A}}
\def\B{\operatorname{B}}
\def\BC{\operatorname{BC}}
\def\C{\operatorname{C}}
\def\D{\operatorname{D}}
\def\E{\operatorname{E}}
\def\F{\operatorname{F}}
\def\G{\operatorname{G}}
\def\GL{\operatorname{GL}}

\def\PO{\operatorname{PO}}

\def\PSp{\operatorname{PSp}}
\def\PSU{\operatorname{PSU}}
\def\PU{\operatorname{PU}}

\def\SL{\operatorname{SL}}
\def\SO{\operatorname{SO}}

\def\SU{\operatorname{SU}}
\def\U{\operatorname{U}}

\def\Ad{\operatorname{Ad}}

\def\Aut{\operatorname{Aut}}

\def\diag{\operatorname{diag}}

\def\det{\operatorname{det}}
\def\dim{\operatorname{dim}}

\def\exp{\operatorname{exp}}

\def\Fix{\operatorname{Fix}}

\def\Hom{\operatorname{Hom}}
\def\id{\operatorname{id}}

\def\Im{\operatorname{Im}}

\def\Int{\operatorname{Int}}

\def\large{\operatorname{large}}
\def\Lie{\operatorname{Lie}}

\def\rank{\operatorname{rank}}
\def\small{\operatorname{small}}
\def\span{\operatorname{span}}
\def\Spin{\operatorname{Spin}}
\def\Stab{\operatorname{Stab}}

\newcommand{\fra}{\mathfrak{a}}

\newcommand{\fre}{\mathfrak{e}}
\newcommand{\frf}{\mathfrak{f}}
\newcommand{\frg}{\mathfrak{g}}
\newcommand{\frh}{\mathfrak{h}}

\newcommand{\frk}{\mathfrak{k}}
\newcommand{\frl}{\mathfrak{l}}

\newcommand{\frt}{\mathfrak{t}}
\newcommand{\fru}{\mathfrak{u}}

\begin{document}

\title[Twisted root system]{Twisted root system of a $(\ast)$-subgroup}
\date{May 2018}
\thanks{}

\author{Jun Yu}
\address{BICMR, Peking University, No. 5 Yiheyuan Road, Haidian District, Beijing 100871, China.}
\email{junyu@bicmr.pku.edu.cn}

\abstract{We classify $(\ast)$-subgroups of compact Lie groups of adjoint type, and associate a twisted
root system to every $(\ast)$-subgroup. We study the structure of twisted root system in several aspects:
properties of the small Weyl group $W_{small}$ and its normal subgroups $W_{tiny}$ and $W_{f}$; properties
of finite root datum; structure of strips of infinite roots.}
\endabstract

\noindent

\noindent

\maketitle

\tableofcontents

\section{Introduction}

Root system is a fundamental object in Lie theory. It plays an important role in many subjects. We call a
closed abelian subgroup $A$ of a compact Lie group $G$ a $(\ast)$-subgroup if it satisfies the following
condition, \[\dim Z_{G}(A)=\dim A. \tag{$\ast$}\] These particularly include maximal abelian subgroups of
$G$. In this paper we classify $(\ast)$-subgroups of compact semisimple Lie groups of adjoint type, and
associate with every $(\ast)$-subgroup a {\it twisted root system}. This is a generalization of the usual
root system associated to a maximal torus. The main goal of this paper is to study properties of twisted 
root systems. 

\smallskip

Let $A$ be a $(\ast)$-subgroup of a compact Lie group $G$. The twisted root system system is a set 
$R(G,A)$ of roots coming from the action of $A$ on the complexified Lie algebra $\frg$ of $G$. A root 
$\alpha\in R(G,A)$ with $\alpha|_{A^{0}}\neq 0$ is called an infinite root; a root $\alpha\in R(G,A)$ 
with $\alpha|_{A^{0}}=0$ is called a finite root. For an infinite root $\alpha$, we associate with it 
a coroot $\check{\alpha}\in X^{\ast}(A^{0})$; for a finite root $\alpha$, we associate with it a coroot 
group $R^{\vee}(\alpha)\subset\Hom(A/\ker\alpha,\ker\alpha)$. With them, we define reflections 
$s_{\alpha}$ and root transvections $s_{\alpha,\xi}$ ($\xi\in R^{\vee}(\alpha)$). These generate the small 
Weyl group $W_{small}(G,A)$, which is a subgroup of the Weyl group $W_{\large}(G,A)=N_{G}(A)/Z_{G}(A)$.

\smallskip 

Write $R'(G,A)$ for the set of projections to $X^{\ast}(A^{0})$ of infinite roots in $R(G,A)$. An immediate
consequence of the property of twisted root system is the following.

\begin{theorem}\label{T1}
$R'(G,A)$ is a root system in the lattice $X^{\ast}(A^{0})$.
\end{theorem}

We call $A$ a symmetric $(\ast)$-subgroup of $G$ if there is an involutive automorphism
$\theta$ of $G$ such that the Lie algebra of $A$ is a maximal abelian subspace of $$(\frg_0^{A})^{-\theta}
\!=\!\{X\in\frg_0:\Ad(g)X\!=\!X (\forall g\in A)\textrm{ and }\theta(X)\!=\!-X\}.$$ The following is a 
generalization of Theorem \ref{T1}. 

\begin{theorem}\label{T2}
Let $A$ be a symmetric $(\ast)$-subgroup. Then the set $R(G,A^{0})$ of non-zero characters of $A^{0}$ for 
the conjugation action of $A^{0}$ on $\frg$ is a root system in the lattice $X^{\ast}(A^{0})$.
\end{theorem}

Theorems \ref{T1} and \ref{T2} give many tori $B$ in $G$ such that the set of non-zero characters of $B$ 
from the conjugation action of $B$ on $\frg$ is a root system. This has independent interest. 

\smallskip

We define several natural subgroups of $W_{\small}=W_{small}(G,A)$: $W_{tiny}$ (generated by reflections),
$W_{f}$ (generated by root transvections), $W_1$ ($\!=\!\{w\in W_{small}:w|_{A/A^0}\!=\!\id\}$), $W_{0}$ 
($\!=\!\{w\in W_{small}:w|_{A^0}\!=\!\id\}$) and $W'$ ($\!=\!W_1\cap W_0$). These are all normal subgroups 
of $W_{large}(G,A)$. The following statement is nontrivial.  

\begin{theorem}\label{T3}
We have $W_{0}=W_{f}$.
\end{theorem}

\smallskip

We call a finite character $\lambda$ of $A$ a generalized finite root if $\lambda$ is generated by its
multiples $k\lambda$ ($k\in\mathbb{Z}-\{0\}$) with $k\lambda$ a root. We associate with every generalized 
finite root $\lambda$ a coroot group $R^{\vee}(\lambda)$. A counting formula of Vogan showed in 
\cite{Han-Vogan} extends to the setting of $(\ast)$-subgroups.

\begin{theorem}\label{T4}
Let $A$ be a maximal abelian subgroup of $G$, and $\lambda$ be a generalized finite root of $A$ with
order $n$. Then \begin{equation}|R^{\vee}(\lambda)|=\prod_{p^{k}|n, p\textrm{ prime}, k\geq 1}
p^{\dim\frg_{\frac{n}{p^{k}}\lambda}}.\end{equation}
\end{theorem}

By this theorem, only finite roots of order a prime power contribute to coroot groups. Assume that $A$ is 
a finite maximal abelian subgroup, we show that the direct sum of root spaces of finite roots with order 
a prime power is the Lie algebra of a closed subgroup $G'$. A consequence of Theorem \ref{T4} is that the
twisted root systems associated to $A$ as a $(\ast)$-subgroup of $G'$ and as a $(\ast)$-subgroup of $G$ 
are the same. By this, the twisted root system doe not determine the Lie algebra of $G$.

\smallskip

We call an ideal $\frh$ of $\frg$ an $A$-ideal if it is furthermore stable under the conjugation action
of $A$, and we call $\frg$ A-simple if it has no nonzero and proper $A$-ideal. We define two $A$-ideals 
$\frg_{\infty}$ and $\frg_{f}$ of $\frg$ using root spaces of the conjugation action of $A$ on $\frg$. 

\begin{theorem}\label{T5}
Both $\frg_{\infty}$ and $\frg_{f}$ are $A$-ideals of $\frg$, and $$\frg=\frg_{\infty}\oplus\frg_{f}.$$
\end{theorem}

The following theorem shows the importance of the root system $R'(G,A)$ in understanding the $A$-action
on $\frg$.
\begin{theorem}\label{T6}
Suppose $\frg=\frg_{\infty}$. Then, $\frg$ is $A$-simple if and only if $R'(G,A)$ is an irreducible root
system.
\end{theorem}

For an infinite root $\alpha\in (G,A)$, the $\alpha$-strip is the set $$\{\beta\in R(G,A):\beta|_{A^{0}}
\!=\!\alpha|_{A^{0}}\}.$$ Set $$R_{0,\alpha}=\{\beta-\alpha:\beta\in R(G,A),\beta|_{A^{0}}=
\alpha|_{A^{0}}\}.$$ A large portion of this paper is in understanding the structure of strips of 
infinite roots. The following is an interesting result along this direction. 

\begin{theorem}\label{T7}
Suppose $\frg=\frg_{\infty}$, and $R'(G,A)$ is an irreducible root system which is simply-laced, or is
one of the types $\C_{n}$ ($n\geq 3$), $\F_4$, $\G_2$. Then $R_{0,\alpha}=R_{0}(G,A)\cup\{0\}$ for any
short root $\alpha\in R(G,A)-R_{0}(G,A)$, and it is a group. 
\end{theorem}

Here, $R_{0}(G,A)$ is the set of finite roots in $R(G,A)$. In Table 1, we list the root system $R'(G,A)$
and length of strips of infinite roots in $R(G,A)$ for positive-dimensional $(\ast)$-subgroups $A$ in an
adjoint type compact simple Lie group $G$.

\smallskip

The content of this paper is organized as follows. In Section 2, we study $(\ast)$-groups acting on a 
torus by reducing to study $\mathbb{Q}$-representations of abelian groups, and we classify 
$(\ast)$-subgroups in any compact semisimple Lie group of adjoint type. In Section 3, we define twisted 
root system and its small Weyl group. We define five normal subgroups $W_{tiny}$, $W_{f}$, $W_0$, $W_1$, 
$W'$ of the small Weyl group and study their relations and propeties. In Section 4, we extend most 
results in \cite{Han-Vogan} about finite root datum to the setting of $(\ast)$-subgroups. In Section 5, 
we study strips of infinite roots. In Section 6, we discuss some other subjects about twisted root system.

\smallskip

\noindent{\it Notation and conventions.}

In this paper, $G$ always means a compact Lie group. Write $G^{0}$ for the neutral subgroup of
$G$, $\mathfrak{g}_{0}$ for the Lie algebra of $G$, and $\mathfrak{g}=\mathfrak{g}_{0}
\otimes_{\mathbb{R}}\mathbb{C}$ for the complexified Lie algebra of $G$. Let $\pi:G\rightarrow
\Aut(\frg_0)$ be the adjoint homomorphism. In this paper, we also use $\frg_{0}$ to denote the zero
weight space for an abelian group $A$ acting on $\frg$. It is easy for readers to distinguish this 
ambiguity.

Let $\frg_0$ be a compact Lie algebra. Write $z(\frg_0)$ for the center of $\frg_0$, and
$(\mathfrak{g}_{0})_{der}=[\frg_0,\frg_0]$ for the derived subalgebra. Then, $\frg_0=z(\frg_0)\oplus
(\frg_{0})_{s}$.

Let $A$ be a compact abelian group. Write $X^{\ast}(A)=\Hom(A,\U(1))$ for the character group of $A$,
and $X^{\ast}(A)=\Hom(\U(1),A)$ for the cocharacter group of $A$. Write $X^{\ast}_{0}(A)=
X^{\ast}(A/A^{0})$. It is clear that $X_{\ast}(A)=X_{\ast}(A^{0})$, and there is an exact sequence
$$0\rightarrow X^{\ast}_{0}(A)\rightarrow X^{\ast}(A)\rightarrow X^{\ast}(A^{0})\rightarrow 0.$$
As $\Hom(\U(1),\U(1))=\mathbb{Z}\cdot\id$, there is a paring $$X^{\ast}(A)\times X_{\ast}(A)
\rightarrow\mathbb{Z}.$$ For $\lambda\in X^{\ast}(A)$, $t\in X_{\ast}(A)$, write $\lambda(t)$ for the 
pairing between $\lambda$ and $t$.

For any $n\geq 1$, write $\mu_{n}=\{z\in\mathbb{C}^{\times}: z^{n}=1\}$ for the group of $n$-th roots 
of units. For any $n$ and $m$ with $n|m$, there is an injection $\mu_{n}\hookrightarrow\mu_{m}$ and a 
surjection $\mu_{m}\rightarrow\mu_{n}$ given by $$z\mapsto z^{\frac{m}{n}}, \forall z\in\mu_{m}.$$

\smallskip

\noindent{\it Acknowledgement.} A part of this paper was written when the author visited MPI Bonn in 
the summer of 2016. The author would like to thank MPI Bonn for the support and hospitality.

\section{$(\ast)$-subgroups of a compact Lie group}

Let $G$ be a compact Lie group. As in \cite{Yu-maximalE}, we say a closed abelian subgroup $A$ of $G$
satisfying the condition $(\ast)$ if $$\dim\frg_{0}^{A}=\dim A.$$ In this paper, we call these
$(\ast)$-subgroups. Maximal abelian subgroups and Jordan subgroups are $(\ast)$-subgroups.

In this section, we classify $(\ast)$-subgroups in a compact semisimple Lie group of adjoint type, and
study $(\ast)$-groups of automorphisms of a compact torus. In the last, we study $(\ast)$-subgroups $A$
with $X^{\ast}(A)$ generated by two elements.

\subsection{$(\ast)$-groups of automorphisms of a compact semisimple Lie algebra}\label{SS:ast-semisimple}

Let $\fru$ be a compact semisimple Lie algebra, and $G=\Aut(\fru)$. We classify $(\ast)$-subgroups of $G$
through three steps.

\smallskip

\noindent{\it Step 1, reduction to the transitive case.} As $\fru$ is a compact semisimple Lie algebra, it
has a unique decomposition $$\fru=\bigoplus_{1\leq i\leq n}\mathfrak{u}'_{i}$$ into a direct sum of compact
simple Lie algebras $\fru'_{i}$ ($1\leq i\leq n$). Then, the action of $A$ permutes the simple factors
$\fru'_{i}$ ($1\leq i\leq n$). Group $\{\fru'_{i}:1\leq i\leq n\}$ into orbits of $A$ and take the direct
sums of simple factors in each orbit, we may write $\fru$ in a unique way as $$\fru=\bigoplus_{1\leq j
\leq t}\mathfrak{u}_{j}$$ such that each $\fru_{j}$ is stable with respect to the action of $A$ and the
action of $A$ on $\fru_{j}$ permutes the simple factors of it transitively. Then, $$A\subset\prod_{1\le j
\leq t}\Aut(\fru_{j}).$$ It suffices to study the projection of $A$ to $\Aut(\fru_j)$, or to study the
action of $A$ on $\fru_{j}$.

With this reduction, we may assume that $A$ permutes the simple factors of $\fru$ transitively. Then, the
simple factors $\fru'_{i}$ ($1\leq i\leq n$) are isomorphic to each other. Write $\fru_0=\fru'_1$, and
identify each $\fru'_{i}$ with $\fru_0.$ Then, $$\fru=\underbrace{\fru_0\oplus\cdots\oplus\fru_0}_{n},$$
and the action of $A$ on $\fru$ permutes the $n$ simple factors $\fru_0$ transitively.

\smallskip

\noindent{\it Step 2, reduction to the case of a simple Lie algebra.} For $n\geq 1$, write $X=X_{n}=\{1,2,
\dots,n\}$ and identify the symmetric group $S_{n}$ with the permutation group on $X$. Let $n_{1},n_{2},
\dots,n_{s}$ be positive integers satisfying that $n_{i+1}|n_{i}$ for each $i$ ($1\leq i\leq s-1$) and
$n_{1}n_{2}\cdots n_{s}|n$. For each $i$, define $\sigma_{i}: X\rightarrow X$ by $\sigma_{i}(x)=y$
($x,y\in X_{n}$) if and only if $$[\frac{y-1}{n_{1}n_{2}\cdots n_{i}}]=[\frac{x-1}{n_{1}n_{2}\cdots
n_{i}}]$$ and $$y-x\equiv n_{1}n_{2}\cdots n_{i-1}\pmod{n_{1}n_{2}\cdots n_{i}}.$$ Then, each
$\sigma_{i}$ is the product of $\frac{n}{n_{i}}$ $n_{i}$-cycles, such $\sigma_{i}$ ($1\leq i\leq s$)
commutes with each other, and they generate an abelian subgroup of $S_{n}$ isomorphic to $C_{n_1}\times
\cdots\times C_{n_{s}}$. Moreover, $$Z_{S_{n}}(\langle\sigma_1,\dots,\sigma_{s}\rangle)\cong(C_{n_1}
\times\cdots\times C_{n_{s}})^{\frac{n}{n_{1}n_{2}\cdots n_{s}}}\rtimes S_{\frac{n}{n_{1}n_{2}\cdots
n_{s}}},$$ and $\langle\sigma_1,\dots,\sigma_{s}\rangle$ is a transitive subgroup of $S_{n}$ if and
only if $n_{1}n_{2}\cdots n_{s}=n$.

The following statement should be well-known.
\begin{lemma}\label{L:transitive-abelian}
Any transitive abelian subgroup $B$ of $S_{n}$ is of the form $B\cong C_{n_1}\times\cdots\times
C_{n_{s}}$ with generator of $C_{n_{i}}$ the product of $\frac{n}{n_{i}}$ $n_{i}$-cycles, where $n_{1},
n_{2},\dots,n_{s}$ are positive integers satisfying that $n_{i+1}|n_{i}$ ($1\leq i\leq s-1$) and
$n_{1}n_{2}\cdots n_{s}=n$.
\end{lemma}

\begin{proof}
By the classification of finitely generated abelian groups, $B$ can be written in a unique way as $B\cong
C_{n_1}\times\cdots\times C_{n_{s}}$, where $n_{1},n_{2},\dots,n_{s}$ are positive integers satisfying
that $n_{i+1}|n_{i}$ for each $i$ ($1\leq i\leq s-1$). Choose a generator $x_{i}$ of $C_{n_{i}}$. It
suffices to prove that: for each $k$ ($\leq s$), the tuple $(x_1,\dots,x_{k})$ is conjugate to $(\sigma_1,
\dots,\sigma_{k})$ as above. Note that, as $B$ is a transitive subgroup of $S_{n}$, it follows that any
element of $B$ is a product of cycles of equal size. One can show the above statement by an inductive
argument. We omit the details here.
\end{proof}



Let $\fru_0$ be a compact simple Lie algebra. For $n\geq 1$, write $$\fru=\underbrace{\fru_0\oplus\cdots
\oplus\fru_0}_{n}.$$ Let $A$ be a compact abelian subgroup of $\Aut(\fru)$ such that the action of $A$ on
$\fru$ permutes the $n$ simple factors $\fru_0$ transitively. Write $N=\Aut(\fru_0)^{n}\lhd\Aut(\fru),$
and let $K$($\cong S_{n}$) be the permutation subgroup of $\Aut(\fru)$. Then, $\Aut(\fru)=N\rtimes K.$
Write $$\phi:\Aut(\fru)\rightarrow K$$ for the projection map.

\begin{lemma}\label{L:semisimple-simple}
Substituting $A$ by a conjugate subgroup if necessary, one can make $\ker(\phi|_{A})\subset Z_{\Aut(\fru)}
(K)=\Delta(\Aut(\fru_0))$, where $\Delta:\Aut(\fru_0)\rightarrow\Aut(\fru)$ is the diagonal map. Moreover,
with this normalization, $A$ is a $(\ast)$-subgroup of $\Aut(\fru_0)$ if and only $\ker(\phi|_{A})=
\Delta(A')$ with $A'$ a $(\ast)$-subgroup of $\Aut(\fru_0)$.
\end{lemma}

\begin{proof}
Write $B=\phi(A)\subset K\cong S_{n}$. By Lemma \ref{L:transitive-abelian}, we may assume that there exists
positive integers $n_{1},n_{2},\dots,n_{s}$ satisfying that $n_{i+1}|n_{i}$ ($1\leq i\leq s-1$) and $n_{1}
n_{2}\cdots n_{s}=n$, and the tuple $(\sigma_1,\dots,\sigma_{s})$ as defined ahead of Lemma
\ref{L:transitive-abelian} is a set of generators of $B$. Choose $f_{i}\in A$ map to $\sigma_{i}$.
Substituting $(f_1,\dots,f_{s})$ by a conjugate tuple if necessary, we may assume that: if $\sigma_{i}(x)=y$
($x,y\in\{1,\dots,n\}$) with $y>x$, then $f_{i}:\fru_{x}=\fru_0\rightarrow\fru_{y}=\fru_0$ is the identity
map. As $f_{i}$ commutes with all other $f_{j}$ ($1\leq j\leq s$), using the identity maps saying above we
see that: there is a fixed $g_{i}\in\Aut(\fru_0)$ such that if $\sigma_{i}(x)=y$ ($x,y\in\{1,\dots,n\}$)
with $y<x$, then $f_{i}:\fru_{x}=\fru_0\rightarrow\fru_{y}=\fru_0$ is given by $g_{i}\in\Aut(\fru_0)$.
Moreover, $f_{i}$ commutes with $f_{j}$ implies that $g_{i}$ commutes with $g_{j}$.

With the above form of $f_1,\dots,f_{s}$, we have $$\Aut(\fru)^{f_1,\dots,f_{s}}=\Delta(\Aut(\fru_0)^{g_1,
\dots,g_{s}})\cdot\langle f_1,\dots,f_{s}\rangle.$$ This shows the first statement of the conclusion. From
the above form of of $\Aut(\fru)^{f_1,\dots,f_{s}}$, we can write $A$ as the form $$A=\Delta(A')\cdot
\langle f_1,\dots,f_{s}\rangle$$ for some $A'\subset\Aut(\fru_0)$ with $\Delta(A')=\ker(\phi_{A})$. Due to
$f_{i}^{n_{i}}=\Delta(g_{i})$ ($1\leq i\leq s$), we have $g_{1},\dots,g_{s}\in A'.$ Then, $$Z_{\Aut(\fru)}
(A)=\Delta(Z_{\Aut(\fru_0)}(A'))\cdot A.$$ By this, $A$ is a $(\ast)$-subgroup of $\Aut(\fru_0)$ if and only
$A'$ a $(\ast)$-subgroup of $\Aut(\fru_0)$.
\end{proof}

By Lemma \ref{L:semisimple-simple} and its proof, a $(\ast)$-subgroup $A$ of $\Aut(\fru)$ is characterized 
by a tuple $(n_1,\dots,n_{s})$ as in Lemma \ref{L:transitive-abelian}, a $(\ast)$-subgroup $A'$ of 
$\Aut(\fru_0)$ and the choice of a tuple $(g_1,\dots,g_{s})$ contained in $A'$. This reduces the 
classification of $(\ast)$-subgroups of $\Aut(\fru)$ to the classification of $(\ast)$-subgroups of 
$\Aut(\fru_0)$.

\smallskip

\noindent{\it Step 3, simple Lie algebra case.} For any compact simple Lie algebra $\fru_0$, $(\ast)$-subgroups
of $\Aut(\fru_0)$ are classified in \cite{Yu-maximalE}.

\subsection{$(\ast)$-groups of automorphisms of a torus}\label{SS:automorphism-abelian}

Let $\frg_0$ be a real vector space of dimension $n$, and $L$ be a lattice in $\frg_0$. Let $A$ be a compact
abelian group of invertible linear transformations of $\frg_0$ preserving the lattice $L$. Then, $A$ could
be regarded as a subgroup of $\Aut(L)\cong\GL(n,\mathbb{Z})$. Thus, $A$ is a finite group.

Write $V=\span_{\mathbb{Q}}L$. Then, $A\subset\Aut(L)\subset\GL(V)$. Integral lattices and finite
subgroups of $\GL(n,\mathbb{Z})$ are subtle in general. We prefer to consider $A$ as a finite subgroup of
$\GL(V)\cong\GL(n,\mathbb{Q})$. Lemmas \ref{L:integral} and \ref{L:isogeny-lattice} indicate that these
two views are nearly equivalent.

\begin{lemma}\label{L:integral}
If $H$ is a finite group acting on $V=\mathbb{Q}^{n}$, then there exists a lattice $L$ of $V$ such that $L$
is $H$-stable.
\end{lemma}
\begin{proof}
Choose a lattice $L'$ of $V$. Write $$L=\bigcap_{g\in H}gL'.$$ Then, $L$ is an $H$-stable lattice.
\end{proof}

\begin{definition}\label{D:isogeny}
We say two pairs $(A_{1},L_1),(A_{2},L_2)$ of finite groups acting on lattices are isogenous if there is
another pair $(A,L)$ and in injective isogenous (of lattices) $\phi_{i}: L\rightarrow L_{i}$ ($i=1,2$)
such that $(\phi_{i})_{\ast}(A)=A_{i}$. Here $\phi_{\ast}$ means the isomorphism $\GL(L\otimes_{\mathbb{Z}}
\mathbb{Q})\rightarrow\GL(L_{i}\otimes_{\mathbb{Z}}\mathbb{Q})$ induced by $\phi_{i}$.
\end{definition}

Due to the correspondence between tori and lattices in real real vector space, we have a similar definition
for {\it isogeny} between tori with group actions.

\begin{lemma}\label{L:isogeny-lattice}
Let $A$ be a finite group, and $L_1$,$L_2$ be two lattices acting faithfully by $A$. Then $(A,L_1)$ and
$(A,L_2)$ are isogenous if and only if the corresponding $\mathbb{Q}$-representations of $A$ on $L_{1}
\otimes_{\mathbb{Z}}\mathbb{Q}$ and $L_{2}\otimes_{\mathbb{Z}}\mathbb{Q}$ are isomorphic.
\end{lemma}

\begin{proof}
Necessarity is obvious.

Sufficiency. Suppose the $\mathbb{Q}$-representations of $A$ on $L_{1}\otimes_{\mathbb{Z}}\mathbb{Q}$ and
$L_{2}\otimes_{\mathbb{Z}}\mathbb{Q}$ are isomorphic. Then, we can find a $\mathbb{Q}$-representation $V$
of $A$ and $A$-stable lattices $L'_{1}$ and $L'_{2}$ such that $(A,L'_{i})\cong(A,L_{i})$ ($i=1,2$).
Take $L=L'_1\cap L'_{2}$. Then, $(A,L)$ gives a isogeny between $(A,L_1)$ and $(A,L_2)$.
\end{proof}

By Lemma \ref{L:isogeny-lattice}, to classify finite groups acting faithfully on integral lattices (or finite
subgroups of $\GL(n,\mathbb{Z})$) up to isogeny, it is equivalent to classify faithful
$\mathbb{Q}$-representations of finite groups.

\smallskip

For a positive integer $n$, write $\Phi_{n}(t)$ for the degree $n$ cyclotomic polynomial, write $a$ for a
generator of $C_{n}$. We define the cyclotomic $\mathbb{Q}$-module $V_{n}$ of $C_{n}$ by $$V_{n}=
\mathbb{Q}[t]/(\Phi_{n}(t)),\quad a\cdot [f(t)]=[tf(t)](\forall f(t)\in\mathbb{Q}[t]).$$ Note that, as
$\Phi_{n}(t)$ is a $\mathbb{Q}$-irreducible polynomial, $V_{n}$ is actually an extension field of $\mathbb{Q}$,
i.e., the degree $n$ cyclotomic field. We have $\dim_{\mathbb{Q}}V_{n}=\phi(n)$, where $\phi$ is Euler
$\phi$ function.

\begin{lemma}\label{L:cyclotomic}
For any $n\geq 1$, the degree $n$ cyclotomic module $V_{n}$ of $C_{n}$ is irreducible, and $$Z_{\GL(V_{n})}
(C_{n})=V_{n}^{\times},$$ where $V_{n}^{\times}$ means the multiplicative group of non-zero elements in the
field $V_{n}$.
\end{lemma}

\begin{proof}
The irreducibility follows from the $\mathbb{Q}$-algebra generated by the multiplication $[f(t)]\mapsto[tf(t)]$
is $V_{n}$, and $V_{n}$ is a field.

Let $g\in Z_{\GL(V_{n})}(C_{n})$. Write $g\cdot[1]=[g_{0}(t)]$. Then, for any $f(t)\in\mathbb{Q}[t]$,
\begin{eqnarray*}
&&g\cdot[f(t)]=g\cdot f(a)\cdot[1]=f(a)\cdot g\cdot[1]\\&=&f(a)\cdot[g_{0}(t)]=[f(t)g_{0}(t)]=[g_{0}(t)f(t)]
\\&=&[g_{0}(t)][f(t)].
\end{eqnarray*}
As $g\in\GL(V_{n})$, we have $[g_{0}(t)]\neq 0$. Thus, $g=[g_{0}(t)]\in V_{n}^{\times}$.
\end{proof}

\begin{lemma}\label{L:rational-decomposition}
Let $V$ be a $\mathbb{Q}$-representation of an abelian group $A$. Then there is a decomposition $$V=
\bigoplus_{1\leq i\leq s}V_{i}$$ such that each $V_{i}$ is $A$-stable and is irreducible.

Let $V$ be a faithful and irreducible $\mathbb{Q}$-representation of an abelian group $A$. Then there exists
$n\geq 1$ such that $A\cong C_{n}$ and $V$ is its cyclotomic module.
\end{lemma}

\begin{proof}
The first statement is basic representation theory.

For the second statement, we first show that for any $a\in A$, the minimal polynomial $f_{a}(t)$ of $a|_{V}$
is a cyclotomic polynomial. Let $n$ be the order of $a$. As $a^{n}=1$, we have $f_{a}(t)| t^{n}-1$. Thus,
$f_{a}(t)$ is a product of $\Phi_{d}(t)$ ($d|n$). Take a $d$ with $\Phi_{d}(t)|f_{a}(t)$. Then, $$V':=
\{v\in V: \Phi_{d}(a)\cdot v=0\}\neq 0.$$ As $A$ is abelain, $V'$ is $A$-stable. Thus, $V'=V$ by the
irreducibility of $V$. Hence, $f_{a}=\Phi_{d}$. Moreover we have $d=n$ as $a$ has order $n$. This shows the
assertion. Secondly, we show that $A$ is a cyclic group. For this, it suffices to show that for any $m\geq 1$,
$A$ has at most $\phi(m)$ elements of order $m$. Let $a$ and $a'$ be two order $m$ elements. Take a character
$\alpha$ of $A$ in $V\otimes_{\mathbb{Q}}\mathbb{C}$. By the first assertion, $\alpha(a)$ and $\alpha(a')$
are both unit elements of order $m$. Thus, there exists $k\in\mathbb{Z}$ with $(m,k)=1$ such that $\alpha(a')
=\alpha(a)^{k}$. Write $b=a'a^{-k}$. Then, $\alpha(b)=1$. By the first assertion, $b=1$. Thus, $a'=a^{k}$.
This shows the second assertion. Thirdly, take a generator $a$ of $A$. Let $n$ be the order of $a$. By the
first assertion, the minimal polynomial of $a|_{V}$ is $\Phi_{n}(t)$. Take any $0\neq v\in V$. Then,
$$\mathbb{Q}[a]\cdot v\cong\mathbb{Q}[t]/(\Phi_{n}(t))=V_{n}.$$ As $V$ is irreducible, we get $V=
\mathbb{Q}[a]\cdot v$. This showes the second statement.
\end{proof}

\begin{lemma}\label{L:cyclotomic-finite}
For any $n\geq 1$, the group of finite order elements in $V_{n}^{\times}$ is isomorphic to $C_{n}$ (or to $C_{2n}$)
if $n$ is even (or odd).
\end{lemma}

\begin{proof}
Note that $V_{n}=V_{2n}$ if $n$ is odd. We may assume that $n$ is even. Take an element $a\in V_{n}$ of order $n$
(for example $a=[t]\in\mathbb{Q}[t]/(\Phi_{n}(t))=V_{n}$). Let $b\in V_{n}$ be a finite order element. We just need
to show $b^{n}=1$. Suppose $b^{n}\neq 1$. Then, the order $m$ of $b$ is not a factor of $n$. Then, there exists a
prime $p$ which has a higher power order in $m$ than in $n$, i.e., there is a integer $k$ such that $p^{k}|n$,
$p^{k+1}\not|n$ and $p^{k+1}|m$. Write $n=p^{k}n'$, $m=p^{k+1}m'$ and set $b'=a^{p^{k}}b^{m'}$. Then, the order
of $b'$ is $p^{k+1}n'=pn$. Thus, $$V_{n}=\mathbb{Q}[a]\supset\mathbb{Q}[b']\cong V_{pn}.$$ Comparing dimension,
we get $\phi(n)\geq\phi(pn)>\phi(n)$. This is a contradiction.
\end{proof}

\noindent{\it $A$-stable lattices and centralizer.} Let $V$ be a faithful $\mathbb{Q}$-representation of a
finite abelian group $A$. From a $A$-stable lattice $L$, one constructs a torus $T$ with a faithful $A$-action.
This gives a compact Lie group $G=T\rtimes A$. Apparently, $A$ is a $(\ast)$-subgroup of $G$ if and only if
there is no nonzero $A$-fixed vectors in $V$. We don't know in general how the centralizer $Z_{\Aut(T)}(A)=
Z_{\GL(L)}(A)$ could be. In another aspect, there are many $A$-stable lattices in $V$, we don't know how to
associate any useful structure on the set of $A$-stable lattices in $V$.

In the case of $A=C_{n}$ and $V=V_{n}$, an $A$-stable lattice in $V_{n}$ is a free $\mathbb{Z}$-submodule $L$
of $V_{n}$, which has rank $\phi(n)$ and has $tL=L$. By Lemma \ref{L:cyclotomic}, $Z_{\GL(V_{n})}(C_{n})=
V_{n}^{\times}$ is abelian. Thus, $Z_{\GL(L)}(A)$ is also an abelian group. If $L$ is further a
$\mathbb{Z}$-subalgebra of $V_{n}$, then $$Z_{\GL(L)}(A)=L^{\times}$$ is the group of invertible elements
in $L$. A particular example is $L=\mathbb{Z}[t]/(\Phi_{n}(t))$.

By Lemma \ref{L:cyclotomic-finite}, the group of finite order elements in $V_{n}^{\times}$ in isomorphic to
$C_{n}$ (or $C_{2n}$ when $n$ is odd). Thus, the subgroup of finite order elements in $Z_{\GL(L)}(A)$ is equal
to $\langle A,-I\rangle$.

\begin{question}\label{Q:cyclotomic}
Is $Z_{\GL(L)}(A)$ itself a finite group?
\end{question}

If the answer to Question \ref{Q:cyclotomic} is affirmative, then $Z_{\GL(L)}(A)=\langle A,-I\rangle$. Thus,
$A$ is a maximal subgroup of $\GL(L)$ if $n$ is even.

\subsection{$(\ast)$-groups with a $2$-generated character group}\label{SS:ast-TwoGenerated}

Let $A$ be a compact abelian group with character group generated by two elements. In the regular case, 
$A$ is isomorphic to the product of two finite cyclic groups, the product of a finite cyclic group and 
$\U(1)$, or a two-dimensional torus. In the degenerate case, $A$ is either a finite cyclic group, or 
isomorphic to $\U(1)$. We want to study actions of $A$ on a compact semisimple Lie algebra which satisfy 
the condition $(\ast)$. We may assume that the action is faithful and consider only the regular case. 

\smallskip

\noindent{\it The finite group case.} We first consider the case of a compact simple Lie algebra.
\begin{lemma}\label{L:twogenerated1}
Let $\fru_0$ be a compact simple Lie algebra and $A$ be a finite abelian subgroup of $\Aut(\fru_0)$ satisfying
condition $(\ast)$. If $A$ is generated by two elements, then $\fru_0\cong\mathfrak{su}(m)$ for some $m\geq 2$
and $A$ is an $m$-Hessenberg group.
\end{lemma}

\begin{proof}
This follows from \cite[Proposition 4.1.1]{Borel-Friedman-Morgan}.
\end{proof}

For a compact semisimple Lie algebra $\fru$, by the reduction in Subsection \ref{SS:ast-semisimple}, we
first reduce to the case that $A$ permutes the simple factors of $\fru$ transitively. By Lemmas
\ref{L:transitive-abelian} and Lemma \ref{L:semisimple-simple}, the image of $A$ in $S_{n}$ is either
isomorphic to $C_{n}$, or to $C_{n_1}\times C_{n_2}$ with $1\neq n_{2}|n_{1}$ and $n=n_1n_{2}$. Moreover
both reduce to the compact simple Lie algebra case, and then are solved by Lemma \ref{L:twogenerated1}

\smallskip

\noindent{\it The one-dimensional group case.} By the reduction in Subsection \ref{SS:ast-semisimple}, we
first reduce to the case that $A$ permutes the simple factors of $\fru$ transitively. In the transitive case,
by Lemmas \ref{L:transitive-abelian} and \ref{L:semisimple-simple}, the image of $A$ in $S_{n}$ must be isomorphic
to $C_{n}$. Then, it reduces to the compact simple Lie algebra case. In the compact simple Lie algebra case it
is solved by the following lemma .

\begin{lemma}\label{L:TwoGenerated2}
Let $\fru_0$ be a compact simple Lie algebra and $A$ be a compact abelian subgroup of $\Aut(\fru_0)$ satisfying
condition $(\ast)$. If $\dim A=1$ and $A/A^{0}$ is cyclic, then $\fru_0\cong\mathfrak{su}(2)$ or $\mathfrak{su}(3)$.
When $\fru_0\cong\mathfrak{su}(2)$, $A$ is maximal torus of $\Aut(\fru_0)\cong\PSU(2)$; when $\fru_0\cong
\mathfrak{su}(3)$, $A=\langle\sigma,A^{0}\rangle$, where $\sigma$ is an outer involution and $A^{0}$ is a
maximal torus of $\Int(\fru_0)^{\sigma}$.
\end{lemma}

\begin{proof}
Write $A=A^{0}\times\langle\theta\rangle\subset\Aut(\fru_0)$ with $\theta$ a finite order automorphism $\theta$
of $\fru_0$. Then, $A^{0}$ must be a maximal torus of $(\Aut(\fru_0)^{\theta})^{0}$. As $\dim A=1$, we have
$\rank\fru_{0}^{\theta}=1$. By \cite[Chapter IV, Proposition 4.2]{BtD}, $\rank\fru_0^{\theta}$ depends only
on the coset $\theta\Int(\fru_0)$ in $\Aut(\fru_0)$. Using diagram automorphism, one can calculate $\rank
\fru_0^{\theta}$ in all cases. Only two cases may happen if $\rank\fru_{0}^{\theta}=1$: (1), $\fru_0\cong
\mathfrak{su}(2)$; (2), $\fru_0\cong\mathfrak{su}(3)$ and $\theta$ is an outer automorphism. The conclusion
follows then.
\end{proof}

\smallskip

\noindent{\it The two-dimensional group case.} If $A$ is a two-dimensional torus, then $\fru$ is a rank two compact
semisimple Lie algebra and $A$ is a maximal torus of $\Int(\fru)$.

\section{Twisted root system associated to a $(\ast)$-subgroup}\label{S:ast-TRS}

Let $G$ be a compact Lie group and $A$ be a $(\ast)$-subgroup of $G$. The conjugation action of $A$ on the
complexified Lie algebra $\frg$ of $G$ gives a decomposition $$\frg=\sum_{\lambda\in X^{\ast}(A)}\frg_{\lambda},$$
where $$\frg_{\lambda}=\{Y\in\frg: \Ad(a)Y=\lambda(a)Y, \forall a\in A\}.$$ Since $A$ satisfies the condition
$(\ast)$, the zero-weight space $\frg_{0}=\fra$. Set $$R(G,A)=\{\alpha\in X^{\ast}(A)-\{0\}: \frg_{\alpha}\neq
0\}.$$ We call an element $\alpha\in R(G,A)$ a {\it root}, and call $\frg_{\alpha}$ the {\it root space} of
$\alpha$. A root $\alpha$ is called an {\it infinite root} if $\alpha|_{A^{0}}\neq 0$; it is called a
{\it finite root} if $\alpha|_{A^{0}}=0$. Write $R_{0}(G,A)$ for the set of finite roots.

Set $$W_{large}(G,A)=N_{G^{0}}(A)/Z_{G^{0}}(A),$$ and call it the {\it Weyl group} of $A$. By conjugation,
$W_{large}(G,A)$ acts on $A$ faithfully. It is clear that the induced action of $W_{large}(G,A)$ on $X^{\ast}(A)$
preserves $R(G,A)$. It is easy to show that $W_{large}(G,A)$ is a finite group.

Endow $G$ with a biinvariant Riemannian metric. Restricted to $A$, it gives a positive-definite inner product
on the Lie algebra $\fra_0$ of $A$, which is $W_{large}(G,A)$ invariant. It induces a positive-definite inner
inner on the dual space $\fra_{0}^{\ast}=\Hom_{\mathbb{R}}(\fra_0,\mathbb{R})$. Since $X^{\ast}(A^{0})\subset
i\fra_{0}^{\ast}$, multiplying by $-1$ and by restriction it gives a positive-definite inner product on
$X^{\ast}(A^{0})$. Without ambiguity, we write $(\cdot,\cdot)$ for each of these inner products. For
$\alpha,\beta\in X^{\ast}(A)$, set $$(\alpha,\beta)=(\alpha|_{A^{0}},\beta|_{A^{0}}).$$

\subsection{Infinite roots}\label{SS:infiniteRoot}

For an infinite root $\alpha$, write $$G^{[\alpha]}=Z_{G}(\ker\alpha),\quad \frg^{[\alpha]}=
Z_{\frg}(\ker\alpha).$$ It is clear that $\frg^{[\alpha]}$ is the complexified Lie algebra of $G^{[\alpha]}$.

\begin{lemma}\label{L:infinity1}
The derived subalgebra of $\frg^{[\alpha]}$ is isomorphic to $\mathfrak{sl}_{2}(\mathbb{C})$.
\end{lemma}

\begin{proof}
As $\alpha|_{A^{0}}\neq 1$, we have $A=A^{0}\cdot\ker\alpha$. Apparently, $A$ is also a $(\ast)$-subgroup of
$G^{[\alpha]}$. By definition, $\ker\alpha\subset Z(G^{[\alpha]})$. Thus, $$Z_{\frg^{[\alpha]}}(A^{0})=\fra.$$
Hence, $A^{0}$ is a maximal torus of $(G^{[\alpha]})^{0}$. For a root $\beta\in R(G,A)$, $\beta|_{\ker\alpha}
=0$ if and only if $\beta\in\mathbb{Z}\alpha$. Thus, $$\frg^{[\alpha]}=\fra\oplus(\bigoplus_{k\in\mathbb{Z}-
\{0\}}\frg_{k\alpha}).$$ Write $$\frg^{[\alpha]}=(\frg^{[\alpha]})_{der}\oplus z(\frg^{[\alpha]})$$ for the
direct sum decomposition of $\frg^{[\alpha]}$ into its semisimple part and central part. Then,
$z(\frg^{[\alpha]})$ is equal to the complexified Lie algebra of $\ker\alpha$, and $(\frg^{[\alpha]})_{der}$
is a semisimple Lie algebra of rank $1$. Hence, $(\frg^{[\alpha]})_{der}\cong\mathfrak{sl}_{2}(\mathbb{C}).$
\end{proof}

Write $G_{[\alpha]}$ for the connected Lie subgroup of $G$ with Lie algebra $\frg_0\cap(\frg^{[\alpha]})_{der}$.
Put $$T_{[\alpha]}=G_{[\alpha]}\cap A.$$ By the proof of Lemma \ref{L:infinity1}, $T_{[\alpha]}$ is a maximal
torus of $G_{[\alpha]}$ and \begin{equation}\label{Eq:split-infinite}A=T_{[\alpha]}\cdot\ker\alpha.
\end{equation} Due to $(\frg^{[\alpha]})_{der}\cong\mathfrak{sl}_{2}(\mathbb{C})$, there is a unique
$\check{\alpha}\in X_{\ast}(A)$ whose image lies in $T_{[\alpha]}$ and satisfies $\alpha(\check{\alpha})=2$.

Define the {\it reflection} $s_{\alpha}: A\rightarrow A$ by \begin{equation}\label{Eq:reflection}s_{\alpha}(a)
=a\check{\alpha}(\alpha(a))^{-1},\ \forall a\in A.\end{equation} It is clear that $s_{\alpha}|_{\ker\alpha}=1$
and $s_{\alpha}|_{\Im\check{\alpha}}=-1$. Since $A=(\ker\alpha)\cdot\Im\check{\alpha}$, we have $s_{\alpha}^2
=\id$. The induced action of $s_{\alpha}$ on $X^{\ast}(A)$ is given by \[s_{\alpha}(\lambda)=\lambda-
\lambda(\check{\alpha})\alpha,\ \forall\lambda\in X^{\ast}(A).\]

\begin{lemma}\label{L:reflection}
There exists an element $n_{\alpha}\in N_{G_{[\alpha]}}(A)$ such that $$\Ad(n_{\alpha})|_{A}=s_{\alpha}.$$
\end{lemma}
\begin{proof}
As $G_{[\alpha]}$ is a connected compact Lie group of rank one, we have $G_{[\alpha]}\cong\SU(2)$ or $\SO(3)$.
By direct calculation, one finds $n_{\alpha}\in N_{G_{[\alpha]}}(T_{[\alpha]})$ such that
$\Ad(n_{\alpha})|_{T_{[\alpha]}}=-1.$  As $G_{[\alpha]}$ commutes with $\ker\alpha$ and $A=T_{[\alpha]}\cdot
\ker\alpha$, we get $n_{\alpha}\in N_{G_{[\alpha]}}(A)$ and $\Ad(n_{\alpha})|_{A}=s_{\alpha}.$
\end{proof}

\begin{lemma}\label{L:infinity2}
For any $\lambda\in X^{\ast}(A)$, we have $$\lambda(\check{\alpha})=\frac{2(\lambda,\alpha)}{(\alpha,\alpha)}.$$
\end{lemma}
\begin{proof}
As the inner product $(\cdot,\cdot)$ $X^{\ast}(A)$ is induced from a biinvariant Riemannian metric on $G$,
it is $W_{large}(G,A)$ invariant. By Lemma \ref{L:reflection}, $s_{\alpha}\in W_{large}(G,A)$. Thus, it
is $s_{\alpha}$ invariant. Hence, for any $\lambda\in X^{\ast}(A)$, \begin{eqnarray*}&&(\lambda,\alpha)
\\&=&(s_{\alpha}(\lambda),s_{\alpha}(\alpha))\\&=&(\lambda-\lambda(\check{\alpha})\alpha,-\alpha)\\&=&
-(\lambda,\alpha)+\lambda(\check{\alpha})(\alpha,\alpha).\end{eqnarray*} Therefore, $\lambda(\check{\alpha})
=\frac{2(\lambda,\alpha)}{(\alpha,\alpha)}.$
\end{proof}

\begin{lemma}\label{L:infinity3}
Let $A$ be a $(\ast)$-subgroup of a compact Lie group $G$, and $\alpha\in R(G,A)-R_{0}(G,A)$ be an infinite
root. Then $\dim\frg_{\alpha}=1$, $-\alpha$ is a root, and $k\alpha$ is not a root for any $k\in\mathbb{Z}-
\{0,\pm{1}\}$. Moreover, the action of $s_{\alpha}$ on $A$ preserves infinite roots and finite roots.
\end{lemma}

\begin{proof}
From the proof of Lemma \ref{L:infinity1}, we see that the weight space $\frg_{k\alpha}$ ($k\in\mathbb{Z}-
\{0\}$) is contained in $(\frg^{[\alpha]})_{der}$. As $(\frg^{[\alpha]})_{der}\cong\mathfrak{sl}_{2}(\mathbb{C})$,
we get that $\dim\frg_{\alpha}=1$, $-\alpha$ is a root, and $k\alpha$ is not a root for any $k\in\mathbb{Z}-
\{0,\pm{1}\}$. As the action of $W_{large}(G,A)$ on $A$ preserves infinite roots and finite roots, so does
$s_{\alpha}\in W_{large}(G,A)$.
\end{proof}

\smallskip

Write $$R'(G,A)=\{\alpha|_{A^{0}}:\alpha\in R(G,A)-R_{0}(G,A)\}.$$ For $\alpha'\in R'(G,A)$, choose $\alpha
\in R(G,A)-R_{0}(G,A)$ such that $\alpha|_{A^{0}}=\alpha'$. Set $$\check{\alpha'}=\check{\alpha}.$$ By Lemma
\ref{L:infinity2}, this does not depend on the choice of $\alpha$. We call the set $R'(G,A)$ together with
the map $\alpha'\mapsto\check{\alpha'}$ ($\alpha'\in R'(G,A)$) the {\it restricted root system} of $A^{0}$
in $G$. For any $\alpha'\in R'(G,A)$, we have a reflection $s_{\alpha'}$ on $A^{0}$ defined by
\[s_{\alpha'}(a)=a\check{\alpha'}(\alpha'(a))^{-1},\ \forall a\in A^{0}.\] It induces a reflection on
$X^{\ast}(A^{0})$ given by \[s_{\alpha'}(\lambda)=\lambda-\lambda(\check{\alpha'})\alpha',\
\forall\lambda\in X^{\ast}(A^{0}).\]

Recall that we defined root system in a lattice in \cite[Definition 2.2]{Yu-dimension}. Besides the structure
of usual root system, the extra condition is ``strong integrality'': for a root system $\Phi$ in a lattice
$L$, it is required that $\frac{2(\lambda,\alpha)}{(\alpha,\alpha)}\in\mathbb{Z}$ for any $\alpha\in\Phi$ and
any $\lambda\in L$.

\begin{theorem}\label{T:restrictedRS1}
$R'(G,A)$ is a root system in the lattice $X^{\ast}(A^{0})$.
\end{theorem}

\begin{proof}
We give two proofs for this. First proof: by Lemma \ref{L:reflection}, the set $R'(G,A)$ is invariant under
$s_{\alpha'}$ for any $\alpha'\in R'(G,A)$. By Lemma \ref{L:infinity2}, any $\alpha'\in R'(G,A)$ satisfies
strong integrality. Thus, $R'(G,A)$ is a root system in the lattice $X^{\ast}(A^{0})$.

Second proof: by \cite[Corollary 3.4]{Yu-dimension}, the set $\Psi'_{A^{0}}$ of the union of root systems
of connected closed subgroups $H$ of $G$ with $A^{0}$ a maximal torus of $H$ is a root system in
$X^{\ast}(A^{0})$. By Lemma \ref{L:infinity1}, we have $\Psi'_{A^{0}}=R'(G,A).$ Thus, $R'(G,A)$ is a
root system in the lattice $X^{\ast}(A^{0})$.
\end{proof}

\subsection{Finite roots and coroot groups}

For a finite root $\alpha$ of order $n$, write $$G^{[\alpha]}=Z_{G}(\ker\alpha).$$ Its complexified Lie
algebra is $$\frg^{[\alpha]}=\fra\oplus\bigoplus_{m\in\mathbb{Z}-\{0\}}\frg_{m\alpha}.$$ It is clear that
$A/\ker\alpha\cong\mathbb{Z}/n\mathbb{Z}$. Choose an element $y\in A$ generate $A/\ker\alpha$.

\begin{lemma}\label{L:abelian1}
The subalgebra $\frg^{[\alpha]}$ is abelian.
\end{lemma}
\begin{proof}
Write $\frg^{[\alpha]}=z(\frg^{[\alpha]})\oplus(\frg^{[\alpha]})_{der}$ for the direct sum decomposition
of $\frg^{[\alpha]}$ into its central part and semisimple part. Since $\alpha$ is a finite root, we have
$A^0\subset\ker\alpha$. Thus, $\fra\subset z(\frg^{[\alpha]})$. As $y\in A$ generates $A/\ker\alpha$, we
have $\fra=Z_{\frg}(A)=(\frg^{[\alpha]})^{y}$. Hence, $(\frg^{[\alpha]})_{der}^{y}=0$. By a theorem of Borel,
this implies  $(\frg^{[\alpha]})_{der}=0$. Thus, $\frg^{[\alpha]}$ is abelian.
\end{proof}



Write $$\frg_{[\alpha]}=\sum_{m\in\mathbb{Z}-\{0\}}\frg_{m\alpha}.$$ Let $G_{[\alpha]}$ be the connected
Lie subgroup of $G$ with Lie algebra $\frg_0\cap\frg_{[\alpha]}$.

\begin{lemma}\label{L:closed1}
$G_{[\alpha]}$ is a torus and $G_{[\alpha]}\cap A$ is a finite group.
\end{lemma}

\begin{proof}
By Lemma \ref{L:abelian1}, $(G^{[\alpha]})^0$ is a torus. The conjugation action of $y$ on
$(G^{[\alpha]})^0$ is given by some $y_{\ast}\in\Aut((G^{[\alpha]})^0)$. Thus, $(\Fix y_{\ast})^0=A^0$
due to $A$ satisfies the condition $(\ast)$. As the order of $\alpha$ is $n$, the order of $y_{\ast}$
is also $n$. Write $$z_{\ast}=I+y_{\ast}+\cdots+y_{\ast}^{n-1}.$$ Then, $z_{\ast}$ is an endomorphism of
$(G^{[\alpha]})^0$ and $G_{[\alpha]}$ is the neutral subgroup of $\{x\in (G^{[\alpha]})^0: z_{\ast}(x)
=1\}.$ Thus, $G_{[\alpha]}$ is a closed subgroup of $G$. As $\frg^{[\alpha]}$ is abelian, so is
$\frg_{[\alpha]}$. Thus, $G_{[\alpha]}$ is abelian. By definition, $G_{[\alpha]}$ is connected.
Hence, it is a torus.

Since $\frg_{[\alpha]}\cap\fra=0$, it follows that $G_{[\alpha]}\cap A$ is a finite group.
\end{proof}


Write $$R^{\vee}(\alpha)=\Hom(A/\ker\alpha,A\cap G_{[\alpha]}).$$ For any $\xi\in R^{\vee}(\alpha)$,
define $s_{\alpha,\xi}:A\rightarrow A$ by $$s_{\alpha,\xi}(a)=a\xi(a),\ \forall a\in A.$$


\begin{lemma}\label{L:transvection1}
For any $\xi\in R^{\vee}(\alpha)$, there exists $n\in N_{G_{[\alpha]}}(A)$ such that $\Ad(n)|_{A}=s_{\alpha,
\xi}.$
\end{lemma}

\begin{proof}
Due to $\frg_{[\alpha]}\cap\fra=0$, $\Ad(y^{-1})-I$ is an endomorphism of $\frg_{[\alpha]}$ without eigenvalue
$0$. Thus, the endomorphism $$\Ad(y^{-1})-I: G_{[\alpha]}\rightarrow G_{[\alpha]}$$ is surjective. Hence, there
exists $n\in G_{[\alpha]}$ such that $y^{-1}nyn^{-1}=\xi(y)$. From this, $nyn^{-1}=y\xi(y)$. As $G_{[\alpha]}$
commutes with $\ker\alpha$ and $\xi|_{\ker\alpha}=1$, we get $nxn^{-1}=x\xi(x)$ for any $x\in A$. Therefore,
$n\in N_{G_{[\alpha]}}(A)$ and $\Ad(n)|_{A}=s_{\alpha,\xi}.$
\end{proof}

By Lemma \ref{L:transvection1}, $s_{\alpha,\xi}$ ($\xi\in R^{\vee}(\alpha)$) is an automorphism of $A$. From
this, the map $R^{\vee}(\alpha)\rightarrow W_{large}(G,A)$ given by $$\xi\mapsto s_{\alpha,\xi}$$ gives an
injection from $R^{\vee}(\alpha)$ to $W_{large}(G,A)$. We call $R^{\vee}(\alpha)$ the {\it coroot group} of
$\alpha$.

\begin{lemma}\label{L:rootSubgroup1}
We have $A\cap G_{[\alpha]}\subset\ker\alpha$.
\end{lemma}

\begin{proof}
Let $x\in A\cap G_{[\alpha]}$. Since $G_{[\alpha]}$ is a torus, it follows that $\Ad(x)$ acts trivially
on $\frg_{k\alpha}$ for any $k\in\mathbb{Z}-\{0\}$. Thus, $(k\alpha)(x)=1$ if $\frg_{k\alpha}\neq 0$.
Particularly we have $\alpha(x)=1$. Hence, $x\in\ker\alpha$. Therefore, $A\cap G_{[\alpha]}\subset
\ker\alpha$.
\end{proof}

\begin{lemma}\label{L:cardinality1}
Let $H$ be a torus and $y$ be an automorphism of $H$ of finite order. For an integer $d$, write $m_{y}(d)$
for the multiplicity of an $d$-th primitive root eigenvalue of the adjoint action of $y$ on $\frh$. Then,
$$|H^{y}|=\prod_{p^{k}|n, p\textrm{ prime}, k\geq 1}p^{m_{y}(p^{k})}.$$
\end{lemma}

\begin{proof}
Write $L=X_{\ast}(H)$ and $U=L\otimes_{\mathbb{Z}}\mathbb{R}.$ Then, $U$ can be identified with the Lie
algebra of $H$. Write $y_{\ast}\in\GL(L)\subset\GL(U)$ for the adjoint action of $y$ on $L$ (and on $U$).
Then, $H^{y}$ can be identified with $(I-y_{\ast})^{-1}(L)/L$. Comparing volumes of $U/L$ and
$U/(I-y_{\ast})^{-1}(L)$, we see that $$|H^{y}|=|(I-y_{\ast})^{-1}(L)/L|=|\det(I-y_{\ast})|.$$ Write
$\Phi_{d}\in\mathbb{Z}[T]$ for the cyclotomic polynomial of degree $d$. Then, the characteristic
polynomial of $y_{\ast}$ is $$f=\prod_{d|n}\Phi_{d}^{m_{y}(d)}.$$ Therefore, $$|\det(I-y_{\ast})|=|f(1)|
=\prod_{d|n}\Phi_{d}(1)^{m_{y}(d)}.$$ For a positive integer $d$, $\Phi_{d}(1)\neq 1$ only when $d=p^{k}$
for some prime $p$ and integer $k\geq 1$. Moreover $\Phi_{p^{k}}(1)=p$ for any $k\geq 1$. Thus, $$|H^{y}|
=|\det(I-y_{\ast})|=\prod_{p^{k}|n, p\textrm{ prime}, k\geq 1}p^{m_{y}(p^{k})}.$$
\end{proof}

\begin{theorem}[Vogan's counting formula]\label{T:Vogan1}
Let $A$ be maximal abelian subgroup of a compact $G$, and $\alpha$ be a finite root of $A$ with order $n$.
Then \begin{equation}\label{Eq3}|R^{\vee}(\alpha)|=\prod_{p^{k}|n, p\textrm{ prime}, k\geq 1}
p^{\dim\frg_{\frac{n}{p^{k}}\alpha}}.\end{equation}
\end{theorem}

\begin{proof}
Since $A$ is a maximal abelian subgroup of $G$, we have $Z_{G}(A)=A$. Thus, $$(G_{[\alpha]})^{y}\subset
(G^{[\alpha]})^{y}=Z_{G}(A)=A.$$ Hence, $(G_{[\alpha]})^{y}\subset A\cap G_{[\alpha]}$. As $A$ is abelian,
we have $A\cap G_{[\alpha]}\subset(G_{[\alpha]})^{y}.$ Therefore, $A\cap G_{[\alpha]}=(G_{[\alpha]})^{y}$.
The conclusion follows from Lemma \ref{L:cardinality1}.
\end{proof}

When $A$ is a finite maximal abelian subgroup, Formula (\ref{Eq3}) is shown in \cite[Eq. (3.7g)]{Han-Vogan}.

\smallskip

We call the set $R(G,A)$ together with the map $\alpha\mapsto\check{\alpha}$ for infinite roots and coroot
groups for finite roots the {\it twisted root system} of $A$ in $G$. Define $W_{small}(G,A)$ as the subgroup
of $W_{large}(G,A)$ generated by $s_{\alpha}$ ($\alpha\in R(G,A)-R_{0}(G,A)$) and $s_{\alpha,\xi}$ ($\alpha
\in R_{0}(G,A)$, $\xi\in R^{\vee}(\alpha)\}$). We call $W_{small}(G,A)$ the {\it small Weyl group} of $A$.
Apparently, the action of $W_{large}(G,A)$ on $A$ preserves infinite roots, the map $\alpha\mapsto
\check{\alpha}$ ($\alpha\in R(G,A)-R_{0}(G,A)$), finite roots, and coroot groups. So does $W_{small}(G,A)$.


\subsection{Weyl group}

Let $A$ be a $(\ast)$-subgroup of a compact Lie group $G$. In the above, we have associated with $A$ a 
twisted root system consisting of a set of roots $R(G,A)$, together with the map of coroots $\alpha
\mapsto\check{\alpha}$ for infinite roots and coroot groups $R^{\vee}(\alpha)$ for finite roots. The 
small Weyl group $W_{small}(G,A)$ is a finite group of automorphisms of $A$ generated by reflections 
$s_{\alpha}$ ($\alpha\in R(G,A)-R_{0}(G,A)$) and transvections $s_{\alpha,\xi}$ ($\alpha\in R_{0}(G,A)$, 
$\xi\in R^{\vee}(\alpha)$).

Set \begin{equation}\label{Eq:W0}W_{0}=\{\gamma\in W_{small}(G,A):\gamma|_{A^0}=\id\},\end{equation}
\begin{equation}\label{Eq:W1}W_{1}=\{\gamma\in W_{small}(G,A):\gamma|_{A/A^0}=\id\},\end{equation}
\begin{equation}\label{Eq:W'}W'=\{\gamma\in W: \gamma|_{A/A^0}=\id,\gamma|_{A^0}=\id\}.\end{equation}
\begin{equation}\label{Eq:Wf}W_{f}=\langle s_{\alpha,\xi}:\alpha\in R_{0}(G,A),\xi\in R^{\vee}(\alpha)
\rangle,\end{equation}
\begin{equation}\label{Eq:Wt}W_{tiny}=\langle s_{\alpha}:\alpha\in R(G,A)-R_{0}(G,A)\rangle,\end{equation}
Apparently, $W_0$, $W_1$, $W'$, $W_{f}$, $W_{tiny}$ are all normal subgroups of $W_{small}(G,A)$.


\begin{proposition}\label{P:W-split}
We have $W=W_{f}W_{tiny},$ $$W_{0}=W_{f}(W_{tiny}\cap W_{0})=W_{f}W',$$
$$W_{1}=W_{tiny}(W_{f}\cap W_{1})=W_{tiny}W',$$
$$W'=W_{0}\cap W_{1}=(W'\cap W_{f})(W'\cap W_{tiny}).$$
\end{proposition}

\begin{proof}
From the definitions it follows that $W=W_{f}W_{tiny}$ and $W'=W_{0}\cap W_{1}$.

For an infinite root $\alpha$, from the definition $$s_{\alpha}(a)=a\check{\alpha}(\alpha(a)^{-1}),
\ \forall a\in A,$$ we see that $s_{\alpha}\in W_1$. Thus, $W_{tiny}\subset W_1$. For a finite root
$\alpha$ and $\xi\in R^{\vee}(\alpha)$, from the definition $$s_{\alpha,\xi}(a)=a\xi(\alpha(a)),
\ \forall a\in A,$$ we see that $s_{\alpha,\xi}\in W_{0}$. Thus, $W_{f}\subset W_{0}$.

It is clear that $W_1\supset W'W_{tiny}$. On the other hand, \begin{eqnarray*}W_1&=&W_1\cap W\\&=&
W_1\cap W_{f}W_{tiny}\\&=&W_{tiny}(W_1\cap W_{f})\\&\subset&W_{tiny}(W_1\cap W_0)\\&=&W_{tiny}W'.
\end{eqnarray*} Thus, $W_1=W'W_{tiny}$. Similarly one can show $W_{0}=W_{f}(W_{tiny}\cap W_{0})=
W_{f}W'$.

Given an element $\gamma\in W'$, write $\gamma=\gamma_1\gamma_2$ for some $\gamma_1\in W_{tiny}$ and
$\gamma_2\in W_{f}$. Then, $$\gamma_1=\gamma\gamma_2^{-1}\in W_{tiny}\cap W'W_{f}\subset W_{tiny}\cap
W_{0}=W_{tiny}\cap W'.$$ In the last equality we use $W_{tiny}\cap W_{0}\subset W_{1}\subset W_{0}=W'$.
Similarly one can show $\gamma_2\in W_{f}\subset W'$. Thus, $W'=(W'\cap W_{f})(W'\cap W_{tiny})$.
\end{proof}



For an infinite root $\alpha$, write $$R_{1,\alpha}=\{\beta\in R(G,A):\beta|_{A^0}=\alpha|_{A^0}\},$$
$$R_{2,\alpha}=\{\beta\in R(G,A):\beta|_{A^0}=2\alpha|_{A^0}\},$$ $$R_{\alpha}=R_{1,\alpha}\cup
R_{2,\alpha}.$$ We call $R_{1,\alpha}$ the {\it $\alpha$-strip}. Write\footnotemark $$R_{0,\alpha}=
\{\beta-\alpha\in X^{\ast}(A):\beta\in R_{1,\alpha}\},$$ $$Z_{0,\alpha}=\{\beta-2\alpha\in X^{\ast}(A):
\beta\in R_{2,\alpha}\}.$$  \footnotetext{By Lemma \ref{L:strip1} below, if $\beta|_{A^0}=
2\alpha|_{A^0}$, then $\beta-\alpha\in R_{1,\alpha}$. Thus, $Z_{0,\alpha}\subset R_{0,\alpha}$.}

Take a simple system $\{\alpha'_{i}: 1\leq i\leq l\}$ of $R'(G,A)$. For each $i$, choose $\alpha_{i}\in
R(G,A)$ such that $\alpha_{i}|_{A^{0}}=\alpha'_{i}$. Write $$B=\bigcap_{1\leq i\leq s}\ker\alpha_{i}.$$
Then, $A=A^{0}\cdot B$. Write $R'$ for the sub-root system generated by $\alpha_1,\dots,\alpha_{l}$.
Then, $R'=\{\alpha\in R(G,A):\alpha|_{B}=1\}$, and the restriction map $$R'\rightarrow R'(G,A),\quad
\beta\mapsto\beta|_{A^{0}}$$ is an isomorphism.

\begin{proposition}\label{P:W0}
We have $W_{small}=W_{0}\rtimes W_{R'}$ and $W_{0}$ is generated by $W_{f}$ and $s_{\beta_1}s_{\beta_2}$
($\beta_1,\beta_{2}\in R_{\beta}$, $\beta\in R'$).
\end{proposition}

\begin{proof}
It is clear that $W_{tiny}$ is generated by $W_{R'}$ and $s_{\beta_1}s_{\beta_2}$ ($\beta_1,\beta_{2}
\in R_{\beta}$, $\beta\in R'$). It is easy to show $s_{\beta_1}s_{\beta_2}\in W'\subset W_{0}$ for any
$\beta_1,\beta_{2}\in R_{\beta}$. Due to $R'\rightarrow R'(G,A)$ is an isomorphism, we have $W_{0}\cap
W_{R'}=\{1\}$. Since $W_{f}\subset W_{0}$ and $W_{0}\cap W_{R'}=\{1\}$, we get $$W_{small}=W_{f}W_{tiny}
=W_{0}\rtimes W_{R'},$$ and $W_{0}$ is generated by $W_{f}$ by $s_{\beta_1}s_{\beta_2}$ ($\beta_1,
\beta_{2}\in R_{\beta}$, $\beta\in R'$).
\end{proof}

\begin{lemma}\label{L:generator}
Let $A$ be a compact abelian group, and $\nu_1,\dots,\nu_{s}\in X^{\ast}(A)$ be linear characters of $A$.
If $\bigcap_{1\leq i\leq s}\ker\nu_{i}=1$, then $\nu_1,\dots,\nu_{s}$ generate $X^{\ast}(A)$.
\end{lemma}

\begin{proof}
Write $\phi: X^{\ast}(A)\rightarrow X^{\ast}(A^{0})$ for the map $$\alpha\mapsto\alpha|_{A^{0}},\ \forall
\alpha\in X^{\ast}(A).$$ Write $X'$ for the subgroup of $X^{\ast}(A)$ generated by $\nu_1,\dots,\nu_{s}$.
By the normalization of abelian subgroups of finitely generated free modules over a PID, we have: either
$\rank X'<\rank X$, or one can find a basis $\{\mu_1,\dots,\mu_{r}\}$ of $X^{\ast}(A^{0})$ and positive
integers $d_{1},\dots,d_{r}$ with $d_{i+1}|d_{i}$ ($1\leq i\leq r-1$) such that $\phi(X')=\span_{\mathbb{Z}}
\{d_{i}\mu_{i}\}$. In the first case, $$\dim(\bigcap_{1\leq i\leq s}\ker\alpha_{i})>0.$$ In the second case,
if $d_1>1$, then $A^{0}\cap\bigcap_{1\leq i\leq s}\ker\alpha_{i}\neq\{1\}$. Thus, $d_{1}=1$. Then, $\phi(X')
=X^{\ast}(A^{0})$. Choose $\lambda_{1},\dots,\lambda_{r}\in X'$ such that $\{\phi(\lambda_1),\dots,
\phi(\lambda_{r})\}$ is a basis of $X^{\ast}(A^{0})$. Write $$B=\bigcap\ker\lambda_{i}.$$ Then, $B$ is a
finite group and $A=A^{0}\times B$. Write $X''=\{\lambda\in X^{\ast}(A): \lambda|_{B}=1\}$. Then, $$X''=
\span_{\mathbb{Z}}\{\lambda_{i}:1\leq i\leq r\}.$$ There is an injection $X'/X''\hookrightarrow X^{\ast}(B)$.
By duality between $B$ and $X^{\ast}(B)$ and the assumption of $\bigcap_{1\leq i\leq s}\ker\nu_{i}=1$, we
get $|X'/X''|=|B|$. Thus, $X'=X^{\ast}(A)$.
\end{proof}

\begin{lemma}\label{L:W0}
Let $\alpha_1$ and $\alpha_2$ be two infinite root. If $\alpha_1|_{A^0}=\alpha_2|_{A^0}$, then $s_{\alpha_1}
s_{\alpha_2}\in W_{f}$; if $\alpha_1|_{A^0}=2\alpha_2|_{A^0}$, then $s_{\alpha_1}s_{\alpha_2}\in W_{f}$ as
well.
\end{lemma}

\begin{proof}
First suppose $\alpha_1|_{A^0}=\alpha_2|_{A^0}$. Write $\gamma=\alpha_{2}-\alpha_{1}$. Put $$G'=Z_{G}
(\ker\alpha_{1}\cap\ker\alpha_{2}).$$ Then, $A$ is a $(\ast)$-subgroup of $G'$. By Lemma \ref{L:generator},
we have $$R(G',A)=\{a\alpha_{2}+b\gamma\in R(G,A):a,b\in\mathbb{Z}\}.$$ Write $\frh_0=(\frg'_{0})_{der}$
for the derived subalgebra of the Lie algebra of $G'$. Set $H=\exp(\frh_0)\subset G'$, and $H'=
\Aut(\frh_0)$. Then, the root spaces $\frg_{\pm{\alpha_{i}}}$ ($i=1,2$) are contained in $\frh$.

There is an adjoint homomorphism $$\phi: G'\rightarrow H'.$$ Write $B=\phi(A)$. Then, $B$ is a
$(\ast)$-subgroup of $H'$. We also view $\alpha_1,\alpha_2$ as roots of $B$ acting on $\frh$. Then,
$\alpha_1,\alpha_2$ generate $X^{\ast}(B)$. Let $n$ be the order of $\gamma$. By the reduction to simple
Lie algebra case (cf. Subsection \ref{SS:ast-semisimple}) and Lemma \ref{L:TwoGenerated2}, we have an
explicit list of $(\frh_0,B)$: \begin{enumerate}
\item $\frh_0=\mathfrak{su}(2)^{n}$, $B=\Delta(T)\times\theta$, where $$\theta(X_1,\dots,X_{n})=(X_2,
\dots,X_{n},X_{1}),\ \forall X_{1},\dots,X_{n}\in\mathfrak{su}(2),$$ $\Delta:\Aut(\mathfrak{su}(2))
\rightarrow\Aut(\mathfrak{su}(2)^{n})$ is the diagonal map, and $T$ is a maximal torus of
$\Int(\mathfrak{su}(2))$.
\item $\frh_0=\mathfrak{su}(3)^{n/2}$, $B=\Delta(T)\times\theta$, where\footnotemark $$\theta(X_1,
\dots,X_{n/2})=(X_2,\dots,X_{n/2},\sigma(X_{1})),\ \forall X_{1},\dots,X_{n/2}\in\mathfrak{su}(3)$$
with $$\sigma(X)=J\overline{X}J^{-1},\quad J=\left(\begin{array}{ccc}0&0&1\\0&1&0\\1&0&0\\\end{array}
\right),$$ \footnotetext{We alternate the choice of $\sigma$ in Lemma \ref{L:TwoGenerated2} so that it
is easier to write the torus $T'$ in the proof of below.}
$\Delta:\Aut(\mathfrak{su}(3))\rightarrow\Aut(\mathfrak{su}(3)^{\frac{n}{2}})$ is the diagonal map,
and $T$ is a maximal torus of $\Int(\mathfrak{su}(3))^{\sigma}$.
\end{enumerate}

Choose an element $y\in A$ such that $\alpha_{2}(y)=1$ and $\gamma(y)=e^{\frac{2\pi i}{n}}$. Then, $y$
generates $A/A^{0}\cdot(\ker\alpha_1\cap\ker\alpha_2)$. By calculation, we have $$s_{\alpha_1}
s_{\alpha_2}|_{A^{0}\cdot(\ker\alpha_1\cap\ker\alpha_2)}=\id$$ and $$s_{\alpha_1}s_{\alpha_2}(y)=
y\check{\alpha_{1}}(e^{\frac{2\pi i}{n}}).$$ From the list of $(\frh_0,B)$ as above, one can verify that
$$\gamma\in R_{0}(G,A).$$ Apparently, $y$ generates $A/\ker\gamma$. It suffices to show
\begin{equation}\label{Eq4}\check{\alpha_{1}}(e^{\frac{2\pi i}{n}})\in A\cap G_{[\gamma]}.\end{equation}
Let $H^{sc}$ be a simply-connected covering of $H$, and $\pi: H^{sc}\rightarrow H$ be the projection.
Write $T^{sc}=((\phi|_{H}\circ\pi)^{-1}T)^{0}$. Then, $\check{\alpha_1}$ lifts to a cocharacter of
$T^{sc}$. Due to $\pi(T^{sc}\cap H^{sc}_{[\gamma]})\subset A^{0}\cap H_{[\gamma]}\subset A\cap
G_{[\gamma]}$, it suffices to show $$\check{\alpha_{1}}(e^{\frac{2\pi i}{n}})\in T^{sc}\cap
H^{sc}_{[\gamma]}.$$

For $(\frh_0,B)$ in the first case, we have $H^{sc}=\SU(2)^{n}$, $T^{sc}=\Delta(T')$, where $T'$ is
the set of diagonal matrices in $\SU(2)$. Moreover, $$\check{\alpha_{1}}(t)=\Delta(\diag\{t,t^{-1}\}),
\ \forall t\in\U(1),$$ and $$H^{sc}_{[\gamma]}=\{(g_1,\dots,g_{n})\in T'^{n}:g_1g_2\cdots g_{n}
=1\}.$$ Thus, $$T^{sc}\cap H^{sc}_{[\gamma]}=\big\{\Delta(g):g\in T',g^{n}=1\big\}.$$ Hence,
$\check{\alpha_{1}}(e^{\frac{2\pi i}{n}})=\Delta(\diag\{e^{\frac{2\pi i}{n}},e^{-\frac{2\pi i}{n}}\})
\in T^{sc}\cap H^{sc}_{[\gamma]}$.

For $(\frh_0,B)$ in the second case, we have $H^{sc}=\SU(3)^{n/2}$, $T^{sc}=\Delta(T')$, where $$T'=
\{\diag\{t,1,t^{-1}\}:t\in\U(1)\}.$$ Moreover, $$\check{\alpha_{1}}(t)=\Delta(\diag\{t^{2},1,
t^{-2}\}),$$ and $$H^{sc}_{[\gamma]}\!=\!\big\{\!(\diag\{u_1,s_1,v_1\},\dots,\diag\{u_{\frac{n}{2}},
s_{\frac{n}{2}},v_{\frac{n}{2}}\}\!)\!:\!s_{j}u_{j}v_{j}\!=\!1,\!\prod_{1\leq j\leq\frac{n}{2}}
\frac{u_{j}}{v_{j}}\!=\!1\big\}.$$ Thus, $$T^{sc}\cap H^{sc}_{[\gamma]}=\big\{\Delta(g):g\in T',g^{n}=1
\big\}.$$ Hence, $\check{\alpha_{1}}(e^{\frac{2\pi i}{n}})=\Delta(\diag\{e^{\frac{4\pi i}{n}},1,
e^{\frac{4\pi i}{n}}\})\in T^{sc}\cap H^{sc}_{[\gamma]}$.

Now suppose $\alpha_1|_{A^0}=2\alpha_2|_{A^0}$. Write $\gamma=\alpha_{1}-2\alpha_{2}$. Do a similar
argument as above. Only Case (2) may happen, and the groups $H^{sc}$, $T^{sc}$, $H^{sc}_{[\gamma]}$
are the same as above. The only difference is $$\check{\alpha_{1}}(t)=\Delta(\diag\{t,1,t^{-1}\})$$
now. Then, $$\check{\alpha_{1}}(e^{\frac{2\pi i}{n}})=\Delta(\diag\{e^{\frac{2\pi i}{n}},1,
e^{\frac{2\pi i}{n}}\})\in T^{sc}\cap H^{sc}_{[\gamma]}.$$
\end{proof}

The following result is the first nontrivial property of twisted root system.
\begin{theorem}\label{T:W0}
We have $W_{0}=W_{f}$ and $W_{small}=W_{f}\rtimes W_{R'}$.
\end{theorem}

\begin{proof}
By Proposition \ref{P:W0}, $W_{0}$ is generated by $W_{f}$ and $s_{\beta_1}s_{\beta_2}$ ($\beta_1,\beta_{2}
\in R_{\beta}$, $\beta\in R'$). By Lemma \ref{L:W0}, $\beta_1,\beta_{2}\in R_{\beta}$ indicates that
$s_{\beta_1}s_{\beta_2}\in W_{f}$. Thus, $W_{0}\subset W_{f}$. Then, $W_{0}=W_{f}$.

By Proposition \ref{P:W0} again, $W_{small}=W_{0}\rtimes W_{R'}=W_{f}\rtimes W_{R'}$.
\end{proof}


\subsection{Bracket relation}\label{SS:TRD-bracket}

For two roots $\alpha_1$ and $\alpha_2$ in $R(G,A)$, let $$G'=G^{[\alpha_1,\alpha_2]}=Z_{G}(\ker\alpha_1
\cap\ker\alpha_2).$$ Then $A$ is a $(\ast)$-subgroup of $G'$. By Lemma \ref{L:generator}, we have
$$R(G',A)=\{a\alpha_{1}+b\alpha_{2}\in R(G,A): a,b\in\mathbb{Z}\}.$$ We can use the action of $A$ on
the complexified Lie algebra of $G'$ to study the bracket relation between $\frg_{\alpha_1}$ and
$\frg_{\alpha_2}$. Write $\frh_0=(\frg'_{0})_{der}$, $H=\exp(\frh_0)$, $H'=\Aut(\frh_0)$, and $B=\phi(A)$,
where $\phi: G'\rightarrow H'$ is adjoint homomorphism. Then, $$X^{\ast}(B)=\span_{\mathbb{Z}}\{\alpha_1,
\alpha_2\}.$$ For any $\gamma\in R(G',A)$, we have $\frg_{[\gamma]}=(\frg_{[\gamma]}\cap\frh)\oplus
(\frg_{[\gamma]}\cap z(\frg'))$. Assume that $[\frg_{[\alpha_1]},\frg_{[\alpha_2]}]\neq 0$. Then,
$$[\frg_{[\alpha_1]},\frg_{[\alpha_2]}]=[\frg_{[\alpha_1]}\cap\frh,\frg_{[\alpha_2]}\cap\frh]\neq 0.$$

By the reduction to simple Lie algebra case (cf. Subsection \ref{SS:ast-semisimple}) and Lemma
\ref{L:TwoGenerated2}, $(\frh_0,B)$ falls into the following list,
\begin{enumerate}
\item The Lie algebra $\frh_0$ is a rank two compact semisimple Lie algebra, and $B$ is a maximal
torus of $\Int(\frh_0)$.
\item $\frh_0=\bigoplus_{1\leq j\leq m}\mathfrak{su}(2)$, and $B=\Delta(B')\rtimes\langle\tau\rangle$,
where $\tau$ acts as a permutation on simple factors of $\frh_0$ and $B'$ is a maximal torus of
$\Int(\mathfrak{su}(2))$.
\item $\frh_0=\bigoplus_{1\leq j\leq m}\mathfrak{su}(3)$, and $B=\Delta(B')\rtimes\langle\tau\rangle$,
where $\tau$ acts as a permutation on simple factors of $\frh_0$ twisted by an outer involutive
automorphism $\sigma$ of $\mathfrak{su}(3)$ and $B'$ is a maximal torus of $\Int(\mathfrak{su}(3))^{\sigma}$.
\item $\frh_0=\mathfrak{su}(n)$ ($n\geq 2$), and $B$ is an $n$-Hessenberg group; or $\frh_0=
\mathfrak{su}(n)^{m}$ ($n\geq 2$, $m\geq 1$), and $B$ is certain twisting of an $n$-Hessenberg group.
\end{enumerate}
Case (1) is well-understood. Cases (2)-(4) are elementary. In Case (4), when $m=1$, for any $\beta_1,
\beta_2\in R(G',A)$, easy calculation shows that $[\frg_{[\beta_1}],\frg_{[\beta_2]}]\neq 0$ if and
only if $\langle\beta_1,\beta_2\rangle$ is not a cyclic group.




\section{Finite root datum}\label{S:FRD}

Let $A$ be a $(\ast)$-subgroup of a compact Lie group $G$. When $A$ is a finite maximal abelian subgroup,
in \cite{Han-Vogan} a finite root datum is constructed from the action of $A$ on the complexified Lie
algebra $\frg$ of $G$. In this section we extend most results of \cite{Han-Vogan} to the more general
setting of $(\ast)$-subgroups, and study the structure of finite root datum.

\subsection{Middle Weyl group}

For any finite character $\lambda$ of $A$, set $$\frg_{[\lambda]}=\bigoplus_{k\in\mathbb{Z}-\{0\}}
\frg_{k\lambda}.$$ We call $\lambda$ a {\it generalized finite root} if $\frg_{[\lambda]}\neq 0$ and
$$\langle\{k\lambda\in R_{0}(G,A):k\in\mathbb{Z}-\{0\}\}\rangle=\mathbb{Z}\lambda.$$


\begin{lemma}\label{L:abelian2}
For any finite character $\lambda$ of $A$, $\frg_{[\lambda]}$ is abelian.
\end{lemma}
\begin{proof}
The proof is the same as the proof for Lemma \ref{L:abelian1}.
\end{proof}

\begin{lemma}\label{L:abelian3}
For two finite roots $\alpha,\beta\in R_{0}(G,A)$, if the orders of $\alpha$ and $\beta$ are coprime to each other, then
$$[\frg_{\alpha},\frg_{\beta}]=0.$$
\end{lemma}
\begin{proof}
Write $\lambda=\alpha+\beta$. As we assume the orders of $\alpha$ and $\beta$ are coprime to each other, we have
$\alpha,\beta\in\mathbb{Z}\cdot\lambda$. Thus, $\frg_{\alpha},\frg_{\beta}\subset\frg_{[\lambda]}$. By Lemma
\ref{L:abelian2}, $\frg_{[\lambda]}$ is abelian. Hence, $[\frg_{\alpha},\frg_{\beta}]=0.$
\end{proof}

Let $G_{[\lambda]}$ be the connected Lie subgroup of $G$ with Lie algebra $\frg_{0}\cap\frg_{[\lambda]}$.

\begin{lemma}\label{L:closed2}
$G_{[\lambda]}$ is a torus and $G_{[\lambda]}\cap A$ is a finite group.
\end{lemma}
\begin{proof}
The proof is the same as the proof for Lemma \ref{L:closed1}.
\end{proof}

Write $$R^{\vee}(\lambda)=\Hom(A/\ker\lambda,G_{[\lambda]}\cap A).$$ For any $\xi\in R^{\vee}(\lambda)$, define
$s_{\lambda,\xi}: A\rightarrow A$ by $$s_{\lambda,\xi}(a)=a\xi(a),\ \forall a\in A.$$

\begin{lemma}\label{L:transvection2}
For any $\xi\in R^{\vee}(\lambda)$, there exists $n\in N_{G_{[\lambda]}}(A)$ such that $$\Ad(n)|_{A}=s_{\lambda,
\xi}.$$
\end{lemma}

\begin{proof}
The proof is the same as the proof for Lemma \ref{L:transvection1}.
\end{proof}

By Lemma \ref{L:transvection2}, $s_{\lambda,\xi}$ ($\xi\in R^{\vee}(\lambda)$) is an automorphism of $A$. From
this, the map $$R^{\vee}(\lambda)\rightarrow W_{large}(G,A),\quad \xi\mapsto s_{\lambda,\xi}$$ gives an injection
from $R^{\vee}(\lambda)$ to $W_{large}(G,A)$. By this, we also use $R^{\vee}(\lambda)$ to denote the subgroup
$\{s_{\lambda,\xi}:\xi\in R^{\vee}(\lambda)\}$ of $W_{large}(G,A)$. We call $R^{\vee}(\lambda)$ the {\it coroot
group} of $\lambda$.

\begin{lemma}\label{L:rootSubgroup2}
We have $A\cap G_{[\lambda]}\subset\ker\lambda$.
\end{lemma}
\begin{proof}
The same as in the proof for Lemma \ref{L:rootSubgroup1}, we have $[k\lambda](a)=1$ for any $a\in A\cap
G_{[\lambda]}$ and any $k\in\mathbb{Z}-\{0\}$ with $k\lambda\in R_{0}(G,A)$. As such $k\lambda$ generate
$\mathbb{Z}\lambda$ by the definition of generalized finite root, we get $\lambda(a)=1$ for any
$a\in A\cap G_{[\lambda]}$. Thus, $A\cap G_{[\lambda]}\subset\ker\lambda$.
\end{proof}

\begin{theorem}[Vogan's counting formula]\label{T:Vogan2}
Let $A$ be maximal abelian subgroup of $G$, and $\lambda$ be a generalized finite root of $A$ with order $n$.
Then \begin{equation}\label{Eq:order5}|R^{\vee}(\lambda)|=\prod_{p^{k}|n, p\textrm{ prime}, k\geq 1}
p^{\dim\frg_{\frac{n}{p^{k}}\lambda}}.\end{equation}
\end{theorem}

\begin{proof}
Since $A$ is a maximal abelian subgroup of $G$, we have $Z_{G}(A)=A$. Choose $y\in A$ generate $A/\ker\lambda$.
Then, $$(G_{[\lambda]})^{y}\subset Z_{G}(A)=A.$$ Hence, $(G_{[\lambda]})^{y}\subset A\cap G_{[\lambda]}$.
As $A$ is abelian, we have $A\cap G_{[\alpha]}\subset(G_{[\alpha]})^{y}.$ Thus, $A\cap G_{[\alpha]}=
(G_{[\alpha]})^{y}$. The conclusion follows from Lemma \ref{L:cardinality1}.
\end{proof}

Define $W_{middle}(G,A)$ as the subgroup of $W_{large}(G,A)$ generated by $s_{\alpha}$ ($\alpha\in R(G,A)-
R_{0}(G,A)\}$ and $\{s_{\lambda,\xi}\}$ ($\lambda$ any generalized finite root, $\xi\in R^{\vee}(\lambda)$).
We call $W_{middle}(G,A)$ the {\it middle Weyl group} of $A$.

\smallskip

\begin{lemma}\label{L:coroot-contain}
Let $\lambda$ and $\lambda'$ be two finite characters of $A$. If $\lambda'$ is a multiple of $\lambda$,
then $R^{\vee}(\lambda')\subset R^{\vee}(\lambda)$.
\end{lemma}

\begin{proof}
In this case $\frg_{[\lambda']}\subset\frg_{[\lambda]}$. Thus, $G_{[\lambda']}\subset G_{[\lambda]}$,
and $R^{\vee}(\lambda')\subset R^{\vee}(\lambda)$.
\end{proof}

We call a generalized finite root $\lambda$ a {\it proper generalized finite root} if it is not the proper
multiple of any other generalized finite root. By Lemma \ref{L:coroot-contain}, it suffices to use infinite
roots and proper generalized finite root in defining $W_{middle}(G,A)$. Note that, a finite root is always
a generalized finite root, but not necessarily a proper generalized finite root.

\begin{question}\label{Q:ast2-ast3}
For a generalized finite root $\lambda$, is it always that $R^{\vee}(\lambda)\subset W_{small}(G,A)$? Equivalently,
is $W_{small}(G,A)=W_{middle}(G,A)$?
\end{question}

\begin{question}\label{Q:coroot-subgroup}
Does it hold that \begin{equation}\label{Eq:saturated}R^{\vee}(\lambda)=\{\xi\in\ker\lambda: s_{\lambda,\xi}
\in W_{middle}(G,A)\}\end{equation} for any generalized finite root $\lambda$?

Does it hold that \begin{equation}\label{Eq:saturated2}R^{\vee}(\alpha)=\{\xi\in\ker\alpha: s_{\alpha,\xi}
\in W_{small}(G,A)\}.\end{equation} for any finite root $\alpha$?
\end{question}


\begin{question}
Suppose $A$ is a maximal abelian subgroup of a compact semisimple Lie group $G$. Is it $R^{\vee}(\alpha)\neq 1$
for any finite root $\alpha\in R_{0}(G,A)$?
\end{question}


\begin{theorem}\label{T:middle-small}
If $A$ is a maximal abelian subgroup of a compact Lie group $G$, then $$W_{middle}(G,A)=W_{small}(G,A).$$
\end{theorem}

\begin{proof}
If suffices to show $$R^{\vee}(\lambda)\subset W_{small}(G,A)$$ for any generalized finite root $\lambda$.
Let $n$ be the order of $\lambda$ and $p_1,\dots,p_{k}$ be all distinct prime factors of $|R^{\vee}(\lambda)|$.
Then, each $p_{i}$ is a prime factor of $n$ by Theorem \ref{T:Vogan2}. For a prime $p_{i}$, let
$p_{i}^{a_{i}}$ be the largest power of $p_{i}$ such that $\frac{m}{p^{a_{i}}}\lambda$ is a root. By Theorem
\ref{T:Vogan2}, the order of the Sylow $p_{i}$-subgroup of $R^{\vee}(\lambda)$ is equal to
$$\sum_{1\leq j\leq a_{i}}\frg_{\frac{m}{p^{j}}\lambda}.$$ Hence, the Sylow $p_{i}$-subgroup of $R^{\vee}
(\lambda)$ is equal to $R^{\vee}(\frac{m}{p^{a_{i}}}\lambda)$, which is contained in $W_{small}(G,A)$. Thus,
$R^{\vee}(\lambda)\subset W_{small}(G,A)$.
\end{proof}

By Theorem \ref{T:middle-small}, Question \ref{Q:ast2-ast3} has an affirmative answer if $A$ is a maximal
abelian subgroup.

\subsection{Saturated $(\ast)$-subgroups}

Let $A$ be an abelian subgroup of a compact Lie group $G$. Set $$Sat(A)=\{g\in G:\Ad(g)|_{\frg_{\alpha}}
\textrm{ is a scalar },\forall\alpha\in X^{\ast}(A)\}.$$ We call $Sat(A)$ the saturation of $A$.
Apparently, $Sat(A)\supset A$. We call $A$ a {\it saturated abelian subgroup} if $Sat(A)=A$.

\begin{lemma}\label{L:saturated1}
Let $G$ be a compact Lie group of adjoint type. Then for any abelian subgroup $A$ of $G$, $Sat(A)$ is still
an abelian group.
\end{lemma}

\begin{proof}
Let $g_1,g_2\in Sat(A)$. By definition, for any character $\alpha$ of $A$, $\Ad(g_{i})|_{\frg_{\alpha}}$
($i=1,2$) is a scalar. Thus, $\Ad([g_1,g_2])=[\Ad(g_1),\Ad(g_2)]=1$. As $G$ is of adjoint type, the map
$\Ad:G\rightarrow\Aut(\frg)$ is injective. Thus, $[g_1,g_2]=1$. Hence, $Sat(A)$ is abelian.
\end{proof}

\begin{proposition}\label{P:saturated2}
Let $G$ be a compact Lie group of adjoint type. Then any maximal abelian subgroup $A$ is saturated.
\end{proposition}

\begin{proof}
By Lemma \ref{L:saturated1}, $Sat(A)$ is abelian. As $A$ is a maximal abelian subgroup, we get $Sat(A)=A$,
i.e., $A$ is saturated.
\end{proof}

\begin{example}\label{E:saturated1}
Let $G=\Spin(7)$ and $$A=\langle-1,e_1e_2e_3e_4,e_1e_2e_5e_6,e_1e_3e_5e_7\rangle\subset G.$$ Then, $A$ is
a maximal abelian subgroup. In this case, $$Sat(A)=\langle-1,e_1e_2,e_1e_3,e_1e_4,e_1e_5,e_1e_6,e_1e_7
\rangle.$$ From this, we see that: (1), the saturation $Sat(A)$ of an $(\ast)$-subgroup is not necessarily
an abelian subgroup; (2), a maximal abelian subgroup is not necessarily a saturated abelian subgroup.
\end{example}

\begin{example}\label{E:saturated2}
Let $G=\Int(\fre_6)$, and $A=F_2\subset G$ be as defined in \cite[Page 238, Eq. (79)]{Yu-maximalE}.
Then, $A$ is a saturated $(\ast)$-subgroup of $G$, but not a maximal abelian subgroup. Thus, a saturated
$(\ast)$-subgroup is not necessarily a maximal abelian subgroup.
\end{example}

\begin{proposition}\label{P:saturated3}
Let $G$ be a compact Lie group of adjoint type, and $A$ be a $(\ast)$-subgroup of $G$. Write $B=Sat(A)$.
Then, $$W_{large}(G,A)=\Stab_{W_{large}(G,B)}(A).$$
\end{proposition}

\begin{proof}
The action of any element $g\in N_{G}(A)$ permutes $A$-eigenspaces. Thus, $g$ normalizes $Sat(A)=B$. This
gives $N_{G}(A)\subset N_{G}(B)$. Any element $g\in Z_{G}(A)$ stabilizes all $A$-eigenspaces. Thus, $g$
commutes with $Sat(A)=B$. This shows $Z_{G}(A)=Z_{G}(B)$. Thus, we have an injection $$W_{large}(G,A)
\hookrightarrow W_{large}(G,B).$$ Apparently, the image is $\Stab_{W_{large}(G,B)}(A)$. Thus,
$W_{large}(G,A)=\Stab_{W_{large}(G,B)}(A).$
\end{proof}

With Proposition \ref{P:saturated3}, when $G$ is of adjoint type, we can calculate the large Weyl group
of a $(\ast)$-subgroup through calculating its saturation and the Weyl group of it. This is our purpose
of defining saturation and saturated abelian subgroup. We remark that based on the classification of
$(\ast)$-subgroups in \cite{Yu-maximalE} and reduction in Subsection \ref{SS:ast-semisimple}, it is not
hard to classify saturated $(\ast)$-subgroups.

When $G$ is not of adjoint type, we do not have yet a generalization of Proposition \ref{P:saturated3}
which enables us to reduce the calculation of large Weyl groups of general $(\ast)$-subgroups to that
of more canonical $(\ast)$-subgroups.

By Proposition \ref{P:saturated3}, we have an injection $$W_{small}(G,A)\hookrightarrow\Stab_{W_{small}
(G,B)}(A).$$ It looks to us that this is not an equality in general.

\begin{question}\label{Q:saturated}
Does Vogan's counting formula in Theorem \ref{T:Vogan2} hold for saturated $(\ast)$-subgroups?
\end{question}

\subsection{Prime decomposition}

In this subsection we always assume $A$ is a finite $(\ast)$-subgroup of a compact Lie group $G$.

Write $R=R(G,A)$. Write $\pi: G\rightarrow\Aut(\frg_0)$ for the adjoint homomorphism of $G$. Let $n$ be
the order of $\pi(A)$. For a prime $p|n$, let $k=k_{p}(A)\geq 1$ be the largest integer such that $p^{k}|n$.
Set $$R_{(p)}=\{\alpha\in R: [p^{k}\alpha]=1\},$$ $$A_{(p)}=\{g\in A:\pi(g)^{\frac{n}{p^{k}}}=1\},$$ and
$$G_{(p)}=Z_{G}(A_{(p)}).$$ Then, the complexified Lie algebra of $G_{(p)}$ is $$\frg_{(p)}=
\bigoplus_{\alpha\in R_{(p)}}\frg_{\alpha}.$$

For a positive integer $m$, set $$R^{(m)}=\{\alpha\in R:[m^{k}\alpha]=1\textrm{ for some }k\geq 1\},$$
$$A^{(m)}=\{g\in A: g^{\frac{n}{\overline{m}}}=1\},$$ where $\overline{m}=\gcd(n,m^{k})$ ($k>>0$). Set
$$G^{(m)}=Z_{G}(A^{(m)}).$$ Then, the complexified Lie algebra of $G^{(m)}$ is $$\frg^{(m)}=
\bigoplus_{\alpha\in R^{(m)}}\frg_{\alpha}.$$ It is clear that each of $R^{(m)}$, $A^{(m)}$, $G^{(m)}$
depends only on the set of prime factors of $\gcd(m,n)$.

By Lemma \ref{L:abelian2}, if $m,m'$ are two coprime positive integers, then $$[\frg^{(m)},\frg^{(m')}]
=0.$$ Particularly for two distinct prime factors $p$ and $q$ of $n$, we have $$[\frg_{(p)},\frg_{(q)}]
=0.$$ By this, $\sum_{p|n}\frg_{(p)}$ is a direct product of Lie algebras. Set $$\frg'=
\sum_{p|n}\frg_{(p)}=\bigoplus_{p|n}\frg_{(p)}.$$ This shows the following theorem.




\begin{theorem}\label{T:pdecom}
$\frg'$ is a reductive subalgebra of $\frg$. It is stable under the conjugation action of $A$, and
$\frg'\cap\frg_0$ is a compact real form of it.
\end{theorem}

\begin{proof}
For each prime $p$, $\frg_{(p)}\cap\frg_{0}$ is the Lie algebra of the compact subgroup $G_{(p)}$. Thus,
$\frg_{(p)}\cap\frg_{0}$ is a compact real form of $\frg_{p}$. Then, $\frg'\cap\frg_0$ is a compact real
form of $\frg'$, and $\frg'$ is a reductive subalgebra.
\end{proof}

\smallskip

Assume $m|n$. Write $n=\prod_{1\leq i\leq s}p_{i}^{a_{i}}$ ($p_{i}\neq p_{j}$, $a_{i}\geq 1$) for the prime 
factorization of $n$. For each positive integer $m|n$, write $m=\prod_{1\leq j\leq t}p_{i_{j}}^{b_{j}}$
($1\leq i_1<\cdots<i_{t}\leq s$, $b_{j}\geq 1$) for the prime factorization of $m$. For each $j$, let 
$m_{j}=\frac{m}{p_{i_{j}}^{b_{j}}}$. Set $$R_{{m}}=R^{(m)}-\bigcup_{1\leq j\leq t}R^{(m_{j})},$$ and 
$$M_{(m)}=\bigoplus_{\alpha\in R_{(m)}}\frg_{\alpha}.$$ Let $m'=\prod_{i\not\in\{i_{j}:1\leq j\leq t\}}
p_{i}.$

\begin{lemma}\label{L:pdecom1}
We have $[\frg^{(m')},M_{(m)}]=0.$ For each prime $p|m$, $[\frg_{(p)},M_{(m)}]\subset M_{(m)}$.
\end{lemma}
\begin{proof}
By Lemma \ref{L:abelian3}, $$[\frg^{(m')},M_{(m)}]=0.$$

Let $\alpha\in R_{(p)}$ and $\beta\in R_{(m)}$. Write $\beta=\beta_1+\beta_2$ with $\beta_1$ a finite
character of order a $p$-power, and $\beta_2$ a finite character of order coprime to $p$. If $\beta_1\neq
-\alpha$, then $\beta_1+\alpha$ is again a finite character of order a $p$-power. Thus, $\alpha+\beta
\not\in R$ or $\alpha+\beta=(\beta_1+\alpha)+\beta_2\in R_{(m)}$. Hence, $$[\frg_{\alpha},\frg_{\beta}]
\subset\frg_{\alpha+\beta}\subset M_{(m)}.$$ If $\beta_1=-\alpha$, then $\alpha=-\beta_{1}$ is a
multiple of $\beta$. Thus, $\frg_{\alpha},\frg_{\beta}\subset\frg_{[\beta]}$. By Lemma \ref{L:abelian2},
$\frg_{[\beta]}$ is abelian. Hence, $[\frg_{\alpha},\frg_{\beta}]=0.$
\end{proof}

Set $$R'_{m}=R-(R^{(m)}\cup R^{(m')})$$ and $$M'_{(m)}=\bigoplus_{\alpha\in R'_{(m)}}\frg_{\alpha}.$$ Then, 
we have a direct sum decomposition \begin{equation}\label{Eq6}\frg=(\frg^{(m)}\oplus\frg^{(m')})\oplus
M'_{(m)}.\end{equation}

\begin{lemma}\label{L:pdecom3}
$\frg^{(m)}\oplus\frg^{(m')}$ is direct product of Lie algebras, and $M'_{(m)}$ is a module of it.
\end{lemma}
\begin{proof}
By Lemma \ref{L:abelian3}, $$[\frg^{(m)},\frg^{(m')}]=0.$$ Then, the first statement follows.

Let $\alpha\in R^{(m)}$ and $\beta\in R'_{(m)}$. Write $\beta=\beta_1+\beta_2$ with $\beta_1$ a finite
character of order coprime to $m'$, and $\beta_2$ a finite character of order coprime to $m$. If $\beta_1
\neq-\alpha$, then $\beta_1+\alpha$ is again a finite character of order coprime to $m'$. Thus,
$\alpha+\beta\not\in R$ or $\alpha+\beta=(\beta_1+\alpha)+\beta_2\in R'_{(m)}$. Hence, $$[\frg_{\alpha},
\frg_{\beta}]\subset\frg_{\alpha+\beta}\subset M'_{(m)}.$$ If $\beta_1=-\alpha$, then $\alpha=-\beta_{1}$
is a multiple of $\beta$. Thus, $\frg_{\alpha},\frg_{\beta}\subset\frg_{[\beta]}$. By Lemma \ref{L:abelian2},
$\frg_{[\beta]}$ is abelian. Hence, $[\frg_{\alpha},\frg_{\beta}]=0.$ This shows $[\frg^{(m)},M'_{(m)}]
\subset M'_{(m)}.$ Similarly one shows $[\frg^{(m')},M'_{(m)}]\subset M'_{(m)}.$
\end{proof}

For each $\emptyset\neq I\subset\{1,\dots,s\}$, write $$m_{I}=\prod_{i\in I}p_{i},$$ and $M_{I}=
M_{(m_{I})}$. Then, we have a direct sum decomposition \begin{equation}\label{Eq5}\frg=
\bigoplus_{I}M_{(I)}.\end{equation}

\begin{lemma}\label{L:pdecom2}
Each $M_{I}$ is a $\frg'$ module, and it is stable under the conjugation action of $A$. If $I\cap J=\emptyset$,
then $[M_{I},M_{J}]=0$.
\end{lemma}
\begin{proof}
By Lemma \ref{L:pdecom1}, if $p|m_{I}$, then $[\frg_{(p)},M_{I}]\subset M_{I}$; if $p\not|m_{I}$, then
$[\frg_{(p)},M_{I}]=0$. Thus, $M_{I}$ is a $\frg'$ module. As $M_{I}$ is a sum of root spaces, it is stable
under the conjugation action of $A$. By Lemma \ref{L:pdecom1}, $[\frg^{(m')},M_{(m_{I})}]=0,$ where $m'=
\prod_{i\not\in I}p_{i}$. If $I\cap J=\emptyset$, then $M_{J}\subset\frg^{(m')}$. Thus, $[M_{I},M_{J}]=0$.
\end{proof}

For each prime $p|n$, $A\subset G_{(p)}$. Apparently, $A$ is a $(\ast)$-subgroup of $G_{(p)}$. Due to
$[\frg_{(p)},\frg_{(q)}]=0$ for primes $p\neq q$, we have $[G_{(p)}^{0},G_{(q)}^{0}]=1$. Set $$G'=\langle
A,G_{(p_{1})}^{0},\dots,G_{(p_{s})}^{0}\rangle.$$ Then, $G'$ is a compact subgroup of $G$, and its Lie
algebra is equal to $\frg'\cap\frg_0$. The tuple $(A,G_{(p)})$ (or $(A,\frg_{(p)})$) gives a
{\it decomposition} of $(A,G')$ (or $(A,\frg')$).

\begin{theorem}\label{T:pdecom2}
If $A$ is a finite maximal abelian subgroup of $G$, then $$W_{small}(G',A)=W_{small}(G,A).$$
\end{theorem}

\begin{proof}
It suffices to show that $A\cap G^{[\lambda]}\subset A\cap G'^{[\lambda]}$ for any finite character
$\lambda$ of $A$. By Vogan's counting formula (Theorem \ref{T:Vogan2}), we have formulas for both,
and all root multiplicities contribute in both formulas are those from finite roots with order a
prime power. As $\frg'$ contains root spaces of such roots, we get $|A\cap G^{[\lambda]}|=
|A\cap G'^{[\lambda]}|$. Thus, $A\cap G^{[\lambda]}=A\cap G'^{[\lambda]}$.
\end{proof}

\begin{question}\label{Q:pdecom1}
Suppose $\frg$ is a simple Lie algebra. Is each $\frg_{(p)}$ (and $\frg^{(m)}$) a simple Lie algebra?
\end{question}

By the description of the possible $(A,\frg)$ in Example \ref{E:pdecom1} below, the above question has
an affirmative answer. However, we prefer a Lie-theoretic proof without using the classification of
$(\ast)$-subgroups. As a corollary, it follows that $\frg'$ is a semisimple Lie algebra.

\smallskip

The finite root data associated to $A$ as a $(\ast)$-subgroup of $G$ (and of $G'$) are the same. Thus,
the finite root datum of $(G,A)$ {\it does not determine the Lie algebra} of $G$.

\begin{example}\label{E:pdecom1}
When $\frg_0$ is a compact simple Lie algebra and $G$ is of adjoint type, there are few examples
of finite $(\ast)$-subgroups $A$ with $|A|$ a composition number, \begin{itemize}
\item[1,] for $n\geq 1$, let $G=\PSU(n)$. Let $A$ be an $n$-Hessenberg group of $G$. Write $$n=
\prod_{1\leq i\leq s}p_{i}^{a_{i}}$$ ($s\geq 2$) for the prime factorization of $n$. In this case,
$$G'\cong\PSU(p_{1}^{a_{1}})\times\cdots\times\PSU(p_{s}^{a_{s}}).$$ For each prime $i$, $$G_{(p_{i})}
=(G_{(p_{i})})^{0}\cdot A,\quad G_{(p_{i})}^{0}\cong\PSU(p_{i}^{a_{i}}),$$ and $A\cap(G_{(p_{i})})^{0}$
is a $p_{i}^{a_{i}}$-Hessenberg group. For a general finite $(\ast)$-subgroup $A$ of $G$, the groups
$G'$ is the same, and $A\cap(G_{(p_{i})})^{0}$ is a finite $(\ast)$-subgroup of $(G_{(p_{i})})^{0}$.
\item [2,] write $\frg_0=\mathfrak{so}(8)$. Note that it has an order 3 outer automorphism $\theta$
such that $\frg_{0}^{\theta}\cong\mathfrak{g}_{2}$. Let $G=\Int(\frg_0)\rtimes\langle\theta\rangle$.
Then $(G^{0})^{\theta}\cong\G_2$. Take a subgroup $B$ of $(\G^0)^{\theta}$ corresponding to the group
$F_1$ ($\cong (C_2)^{3}$) of $\G_2$ as in \cite[Page 234]{Yu-maximalE}. Set $A=B\times\langle\theta
\rangle$. Then, $A$ is a maximal abelian subgroup of $G$, $|A|=3\cdot 2^{3}$. For $A$, $$G_{(2)}=
G^{\theta}\cong\G_2\times\langle\theta\rangle,\quad G_{(3)}=A.$$ Thus, $\frg_{(2)}=\frg^{\theta}$,
and $\frg_{(3)}=0$.
\item[3,] write $\frg_0=\mathfrak{e}_{6}$. Note that it has an outer involutive automorphism $\theta$
such that $\frg_{0}^{\theta}\cong\mathfrak{f}_{4}$. Let $G=\Aut(\frg_0)$. Then $(G^{0})^{\theta}\cong
\F_4$. Take a subgroup $B$ of $(\G^0)^{\theta}$ corresponding to the group $F_2$($\cong(C_3)^{3}$) of
$\F_4$ as in \cite[Page 233]{Yu-maximalE}. Set $A=B\times\langle\theta\rangle$. Then, $A$ is a maximal
abelian subgroup of $G$, $|A|=2\cdot 3^{3}$. For $A$, $$G_{(2)}=A,\quad G_{(3)}=G^{\theta}\cong\F_4
\times\langle\theta\rangle.$$ Thus, $\frg_{(2)}=0$, and $\frg_{(3)}=\frg^{\theta}$.
\item[4,] write $G=\E_8$. Take a subgroup $A$ of $G$ corresponding to the finite $(\ast)$-subgroup
$F_1$($\cong(C_2)^{2}\times(C_{3})^{3}$) of $G$ as in \cite[Page 253]{Yu-maximalE}. Then, $|A|=
2^{2}\cdot 3^{3}$. We have $$G_{(2)}=G_{(2)}^{0}\cdot A,\quad G_{(2)}^{0}\cong\SO(3),$$ $A\cap
G_{(2)}^{0}$ corresponds to the subgroup of diagonal matrices in $\SO(3)$; $$G_{(3)}=G_{(3)}^{0}
\cdot A,\quad G_{(3)}^{0}\cong\F_4,$$ $A\cap G_{(3)}^{0}$ corresponds to a maximal abelian subgroup
of $\F_4$ isomorphic to $(C_3)^{3}$. Thus, $\frg_{(2)}\cong\mathfrak{so}(3,\mathbb{C})$, and
$\frg_{(3)}=\mathfrak{f}_{4}(\mathbb{C})$.
\item[5,] write $G=\E_8$. Take a subgroup $A$ of $G$ corresponding to the finite $(\ast)$-subgroup
$F_2$($\cong(C_2)^{3}\times(C_{3})^{3}$) of $G$ as in \cite[Page 261]{Yu-maximalE}. Then, $|A|=
2^{3}\cdot 3^{3}$. We have $$G_{(2)}=G_{(2)}^{0}\cdot A,\quad G_{(2)}^{0}\cong\G_2,$$ $A\cap
G_{(2)}^{0}$ corresponds to a maximal abelian subgroup of $\G_2$ isomorphic to $(C_2)^{3}$; $$G_{(3)}
=G_{(3)}^{0}\cdot A,\quad G_{(3)}^{0}\cong\F_4,$$ $A\cap G_{(3)}^{0}$ corresponds to a maximal abelian
subgroup of $\F_4$ isomorphic to $(C_3)^{3}$. Thus, $\frg_{(2)}\cong\mathfrak{g}_{2}(\mathbb{C})$,
and $\frg_{(3)}=\mathfrak{f}_{4}(\mathbb{C})$.
\end{itemize}
\end{example}

\begin{example}\label{E:pdecom2}
The following is a list of finite $(\ast)$-groups which are $p$-groups ($p>2$) in a compact simple Lie group
of adjoint type, \begin{itemize}
\item [1,] for $n=p^{a}$ with $p>2$ a prime and $a\in\mathbb{Z}_{>0}$, let $G=\PSU(n)$. In this case $A$
corresponds to a partition of $a$, $a=\sum_{1\leq i\leq k} b_{i}$, $b_{1}\geq\cdots\geq b_{k}\geq 1$,
$$A\cong\bigoplus_{1\leq i\leq k}(C_{p^{b_{i}}})^{2}$$ (\cite[Page 224, Prop. 3.1]{Yu-maximalE}).
\item [2,] for $\fru_0=\frf_4$, $G=\Aut(\fru_0)$, let $A$ be the subgroup $F_2$($\cong(C_3)^{3}$) of $G$
as in \cite[Page 233, Eq. (75)]{Yu-maximalE}.
\item [3,] for $\fru_0=\fre_6$, $G=\Aut(\fru_0)$, let $A$ be the subgroup $F_2$($\cong(C_3)^{3}$) or
$F_3$($\cong(C_3)^{4}$) of $G$ as in \cite[Page 238, Eqs. (79) and (80)]{Yu-maximalE}. Write $G^{sc}$ for
a universal covering of $G$, the pre-images $F_2$ in $G^{sc}$ is also a $(\ast)$-subgroup.
\item [4,] for $\fru_0=\fre_8$, $G=\Aut(\fru_0)$, let $A$ be the subgroup $F_1$($\cong(C_5)^{3}$) or $F_3$
($\cong(C_3)^{5}$) of $G$ as in \cite[Page 261, Eqs. (117) and (119)]{Yu-maximalE}.
\end{itemize}
\end{example}


\section{Strips and FRD}

In this section we always assume that $\frg_0$ is a compact semisimple Lie algebra, and $A$ is a $(\ast)$-subgroup
of $\Aut(\frg_0)$. Write $\frg=\frg_0\otimes_{\mathbb{R}}\mathbb{C}$.

\subsection{$A$-simple Lie algebras and $A$-decompositions}

\begin{definition}\label{D:Aideal}
We call a real ideal $\frh_0$ of $\frg_0$ (or a complex ideal $\frh$ of $\frg$) an $A$-ideal is it is stable
under the conjugation action of $A$.

We call $\frg_0$ (or $\frg$) an $A$-simple Lie algebra if it has no nonzero and proper real (or complex) $A$-ideal.
\end{definition}

The following two statements hold true clearly.

\begin{proposition}\label{P:Asimple1}
$\frg_0$ (or $\frg$) is $A$-simple if and only if the action of $A$ permutes the simple factors of it.
\end{proposition}

\begin{proposition}\label{P:Asimple2}
There is a unique direct sum decomposition of $\frg_0$ (or $\frg$) into $A$-simple ideas, $\frg_{0}=
\bigoplus_{1\leq j\leq s}\frg_{0,j}$ (or $\frg=\bigoplus_{1\leq j\leq s}\frg_{j}$) with $\frg_{j}=\frg_{0,j}
\otimes_{\mathbb{R}}\mathbb{C}$ and $\frg_{0,j}=\frg_{j}\cap\frg_{0}$.
\end{proposition}

By Proposition \ref{P:Asimple1}, $\frg_0$ is $A$-simple if and only if $\frg$ is $A$-simple.



For each $\alpha'\in R'(G,A)$, set $$\frg'_{\alpha'}=\bigoplus_{\alpha\in R(G,A),\alpha|_{A^{0}}=\alpha'}
\frg_{\alpha}.$$ Set $$\frg_{\infty}=(\bigoplus_{\alpha'\in R'(G,A)}\frg'_{\alpha'})\bigoplus(\sum_{\alpha'
\in R'(G,A)}[\frg'_{\alpha'},\frg'_{-\alpha'}]),$$ and let $\frg_{f}$ be the centralizer of $\frg_{\infty}$
in $\frg$.

\begin{theorem}\label{T:decomposition}
Both $\frg_{\infty}$ and $\frg_{f}$ are $A$-ideals of $\frg$, and $$\frg=\frg_{\infty}\oplus\frg_{f}.$$
Moreover, $\frg_0\cap\frg_{\infty}$ (or $\frg_0\cap\frg_{f}$) is a compact real form of $\frg_{\infty}$
(or $\frg_{f}$).
\end{theorem}

\begin{proof}
Write $$\frg'_{0}=(\bigoplus_{\alpha\in R_{0}(G,A)}\frg_{\alpha})\bigoplus\frg_{0}.$$ Then, $$\frg=
(\bigoplus_{\alpha'\in R'(G,A)}\frg'_{\alpha'})\bigoplus\frg'_{0}.$$ Set $\frg'_{\gamma'}=0$ for $\gamma'
\in X^{\ast}(A^{0})-(R'(G,A)\cup\{0\})$. Then, $$[\frg'_{\gamma'},\frg'_{\mu'}]\subset\frg'_{\gamma'+\mu'}$$
for any $\gamma',\mu'\in X^{\ast}(A^{0})$. By this, using Jacobi identity one can show that the adjoint
action of each $\frg'_{\alpha'}$ ($\alpha'\in R'(G,A)$) and of $\frg'_{0}$ on $\frg_{\infty}$ stabilize
it. Thus, $\frg_{\infty}$ is an ideal of $A$. Then, so is $\frg_{f}$. It is clear that $\frg_{\infty}$ is
$A$-stable. Then, so is $\frg_{f}$.

For each $\alpha\in R(G,A)-R_{0}(G,A)$, set $$\mathfrak{sl}_{\alpha}=\frg_{\alpha}\oplus\frg_{-\alpha}
\oplus[\frg_{\alpha},\frg_{-\alpha}].$$ By Lemma \ref{L:infinity1}, $\mathfrak{sl}_{\alpha}\cong
\mathfrak{sl}(2,\mathbb{C})$. One can show that $\mathfrak{sl}_{\alpha}\cap\frg_0$ is a compact real
form of $\mathfrak{sl}_{\alpha}$. Thus, $\frg_0\cap\frg_{\infty}$ is a compact real form of $\frg_{\infty}$.
Taking centralizer, it follows that $\frg_0\cap\frg_{f}$ is a compact real form of $\frg_{f}$.
\end{proof}

\begin{proposition}\label{P:Asimple3}
Assume $\dim A>0$ and $\frg$ is $A$-simple. Then, the infinite root spaces generate $\frg$.
\end{proposition}
\begin{proof}
As $\dim A>0$, we have $R(G,A)\neq R_{0}(G,A)$. Thus, $\frg_{\infty}\neq 0$. Since $\frg$ is $A$-simple
and $\frg_{\infty}$ is an $A$-ideal, we get $\frg_{\infty}=\frg$. This just means: infinite root spaces
generate $\frg$.
\end{proof}

The following result indicates that $R'(G,A)$ is important in studying the $A$-action on $\frg$.

\begin{theorem}\label{T:R'-simple}
Suppose $\frg=\frg_{\infty}$. Then, $\frg$ is $A$-simple if and only if $R'(G,A)$ is an irreducible root
system.
\end{theorem}
\begin{proof}
Sufficiency follows from Proposition \ref{P:Asimple1}.

Necessarity. Suppose $R'(G,A)$ is not irreducible. Write $$R'(G,A)=R'_{1}\bigsqcup R'_2$$ with $R'_1,R'_2$
two nonempty orthogonal sub-root systems. Here orthogonal means for any $\alpha'_{i}\in R'_{i}$ ($i=1,2$),
$\alpha'_1+\alpha'_2$ is not a root. Let $$\frg_{i}=(\sum_{\alpha|_{A^{0}}\in R'_{i}}\frg_{\alpha})\bigoplus
(\sum_{\alpha|_{A^{0}}=\beta|_{A^{0}}\in R'_{i}}[\frg_{\alpha},\frg_{-\beta}])$$ ($i=1,2$). Note that, for
any $\alpha,\beta\in R(G,A)$, $\beta\neq-\alpha$, we have $[\frg_{\alpha},\frg_{\beta}]\subset\frg_{\alpha+
\beta}$ if $\alpha+\beta\in R(G,A)$, and $=\{0\}$ if $\alpha+\beta\not\in R(G,A)$). Using Jacobi identify,
one can show that $\frg_{i}$ is stable under the adjoint action of any $\frg_{\alpha}$ ($\alpha\in R_{0}(G,A)$),
$\frg_{\alpha}$ ($\alpha|_{A^{0}}\in R'_{1}$), $\frg_{\alpha}$ ($\alpha|_{A^{0}}\in R'_{2}$). Thus, $\frg_{i}$
($i=1,2$) is a nonzero ideal of $\frg$. Apparently, they are distinct nonzero and proper ideas and are both
$A$-stable. This is in contradiction with $\frg$ is $A$-simple.
\end{proof}

\subsection{Simply laced root systems of rank $\geq 2$.}

\begin{lemma}\label{L:strip0}
For any $\alpha,\beta\in R(G,A)-R_{0}(G,A)$ with $\beta\in W_{tiny}\alpha$, we have $R_{0,\alpha}=R_{0,\beta}$.
\end{lemma}
\begin{proof}
Write $\beta=w\alpha$ for some $w\in W_{tiny}$. Then, $R_{\beta}=wR_{\alpha}$. As $W_{tiny}$ acts trivially on
$X^{\ast}_{0}(A)$, we get $R_{0,\beta}=R_{0,\alpha}$.
\end{proof}

\begin{lemma}\label{L:strip1}
For any $\alpha,\beta\in R(G,A)-R_{0}(G,A)$ with $(\beta|_{A^{0}},\alpha|_{A^{0}})<0$, we have $\alpha+\beta
\in R(G,A)\cup\{0\}$.
\end{lemma}
\begin{proof}.
By Lemma \ref{L:infinity1}, there is an $\mathfrak{sl}_{2}$-subalgebra $\mathfrak{sl}_{\alpha}=\frg_{\alpha}
\oplus\frg_{-\alpha}\oplus[\frg_{\alpha},\frg_{-\alpha}]$ corresponds to the root $\alpha$. By the theory of
representations of $\mathfrak{sl}_{2}$, $(\beta|_{A^{0}},\alpha|_{A^{0}})<0$ implies that $[\frg_{\alpha},
\frg_{\beta}]\neq 0$. Thus, $\alpha+\beta\in R(G,A)\cup\{0\}$.
\end{proof}

\begin{lemma}\label{L:strip2}
For any $\alpha\in R(G,A)-R_{0}(G,A)$, we have $-\alpha\in R(G,A)-R_{0}(G,A)$, $R_{\alpha}=R_{-\alpha}=
-R_{\alpha}$, and $R_{0,\alpha}-R_{0,\alpha}\subset R_0\cup\{0\}$.
\end{lemma}
\begin{proof}
By Lemma \ref{L:infinity1}, we have $-\alpha\in R(G,A)-R_{0}(G,A)$. We also have $-\alpha'\in R(G,A)-
R_{0}(G,A)$ for any $\alpha'\in R_{\alpha}$. Thus, $R_{-\alpha}=-R_{\alpha}$. Due to $s_{\alpha}(\alpha)
=-\alpha$, by Lemma \ref{L:strip0} we get $R_{-\alpha}=R_{\alpha}$. For any $\alpha_1,\alpha_2 \in
R_{\alpha}$, we have $(-\alpha_{2}|_{A^{0}},\alpha_{1}|_{A^{0}})<0$. By Lemma \ref{L:strip1}, we get
$\alpha_1-\alpha_2\in R_{0}\cup\{0\}$. Thus, $R_{0,\alpha}-R_{0,\alpha}\subset R_0\cup\{0\}$.
\end{proof}

\begin{lemma}\label{L:strip3}
For any $\alpha,\beta\in R(G,A)-R_{0}(G,A)$ with $(\beta|_{A^{0}},\alpha|_{A^{0}})<0$ and $\alpha|_{A^0}+
\beta|_{A^0}\neq 0$, we have $\alpha+\beta\in R(G,A)-R_{0}(G,A)$ and $R_{0,\alpha}+R_{0,\beta}\subset
R_{0,\alpha+\beta}$.
\end{lemma}
\begin{proof}
By Lemma \ref{L:strip1}, $\alpha+\beta\in R(G,A)$. As $\alpha|_{A^0}+\beta|_{A^0}\neq 0$, we have $\alpha
+\beta\in R(G,A)-R_{0}(G,A)$. For any $\alpha'\in R_{\alpha}$ and $\beta'\in R_{\beta}$, we also have
$\alpha'+\beta'\in R(G,A)-R_{0}(G,A)$ by Lemma \ref{L:strip1}. Thus, $R_{0,\alpha}+R_{0,\beta}\subset
R_{0,\alpha+\beta}$.
\end{proof}

\begin{proposition}\label{P:strip4}
For any $\alpha,\beta\in R(G,A)-R_{0}(G,A)$ with $\langle\alpha,\beta\rangle=\langle\beta,\alpha\rangle=-1$,
we have $R_{0,\alpha}=R_{0,\beta}$ and they are subgroups of $X^{\ast}_{0}(A)$.
\end{proposition}

\begin{proof}
In this case $(\beta|_{A^{0}},\alpha|_{A^{0}})<0$ and $\alpha|_{A^0}+\beta|_{A^0}\neq 0$. By Lemmas
\ref{L:strip2} and \ref{L:strip3}, $\gamma=-(\alpha+\beta)\in R(G,A)-R_{0}(G,A)$. As $\langle\alpha,\beta
\rangle=\langle\beta,\alpha\rangle=-1$, we have $\langle\alpha,\gamma\rangle=\langle\beta,\gamma\rangle=-1$.
By Lemmas \ref{L:strip2} and \ref{L:strip3}, we get $$R_{0,\alpha}+R_{0,\beta}\subset R_{0,-\gamma}=
R_{0,\gamma},$$ $$R_{0,\alpha}+R_{0,\gamma}\subset R_{0,-\beta}=R_{0,\beta},$$ $$R_{0,\beta}+R_{0,\gamma}
\subset R_{0,-\alpha}=R_{0,\alpha}.$$ Thus, $R_{0,\alpha}=R_{0,\beta}=R_{0,\gamma}$. As $R_{0,\alpha}=
-R_{0,\alpha}$ by Lemma \ref{L:strip2}, we get $R_{0,\alpha}-R_{0,\alpha}\subset R_{0,\alpha}$. Thus,
$R_{0,\alpha}$ is a subgroup of $X^{\ast}_{0}(A)$.
\end{proof}

\begin{theorem}\label{T:strip1}
Let $R'$ be a simply-laced irreducible sub-root system of $R'(G,A)$ with rank $\geq 2$. Then $R_{0,\alpha}$
($\alpha\in R(G,A)$, $\alpha|_{A^{0}}\in R'$) is constant and it is a subgroup of $X^{\ast}_{0}(A)$.
\end{theorem}

\begin{proof}
Let $\alpha\in R(G,A)$ be with $\alpha|_{A^{0}}\in R'$. Due to $R'$ is simply-laced with rank $\geq 2$,
there exists $\beta\in R(G,A)$ such that $\langle\alpha,\beta\rangle=\langle\beta,\alpha\rangle=-1$. By
Proposition \ref{P:strip4}, $R_{0,\alpha}$ is a subgroups of $X^{\ast}_{0}(A)$.

For any $\beta\in R(G,A)$ with $\beta|_{A^{0}}\in R'$, if $(\beta,\alpha)\neq 0$, then $\langle\alpha,\beta
\rangle=\langle\beta,\alpha\rangle=\pm{1}$. By Proposition \ref{P:strip4}, $R_{0,\beta}=\pm{R_{0,\alpha}}=
R_{0,\alpha}$. If $(\beta,\alpha)=0$, then there exists $\gamma\in R(G,A)$ with $\gamma|_{A^{0}}\in R'$ such
that $$\langle\alpha,\gamma\rangle=\langle\gamma,\alpha\rangle=\langle\beta,\gamma\rangle=\langle\gamma,
\beta\rangle=-1.$$ By Proposition \ref{P:strip4}, $R_{0,\beta}=R_{0,\gamma}=R_{0,\alpha}$.
\end{proof}

\subsection{Rank one root systems}

\begin{lemma}\label{L:strip5}
Let $\alpha\in R(G,A)-R_{0}(G,A)$ and $\lambda\in X^{\ast}_{0}(A)$. If $2\alpha+\lambda\in R(G,A)$, then for any
$k\in\mathbb{Z}$, $\alpha+k\lambda,2\alpha+(2k+1)\lambda\in R(G,A)$.
\end{lemma}

\begin{proof}
Write $\beta=2\alpha+\lambda$. Then, $\check{\alpha}=2\check{\beta}$. For any $a\in A$,
\begin{eqnarray*}s_{\beta}s_{\alpha}(a)&=& s_{\beta}(a\check{\alpha}(\alpha(a)^{-1}))
\\&=&s_{\beta}(a)s_{\beta}(\check{\alpha}(\alpha(a)^{-1}))
\\&=&a\check{\beta}(\beta(a)^{-1})\check{\alpha}(\alpha(a))
\\&=&a\check{\beta}(\beta(a)^{-1}\alpha(a)^{2})
\\&=&a\check{\beta}(\lambda(a)^{-1}).
\end{eqnarray*}
By this $s_{\beta}s_{\alpha}(\alpha)=\alpha+\lambda$. As $s_{\beta}s_{\alpha}\in W_{tiny}$, we have
$s_{\beta}s_{\alpha}(\lambda)=\lambda$. By these, we get $\alpha+k\lambda,2\alpha+(2k+1)\lambda\in R(G,A)$
for any $k\in\mathbb{Z}$.
\end{proof}

\begin{lemma}\label{L:strip6}
Let $\alpha\in R(G,A)-R_{0}(G,A)$ and $\lambda\in X^{\ast}_{0}(A)$. If $\alpha+\lambda\in R(G,A)$, then for any
$k\in\mathbb{Z}$, $\alpha+k\lambda\in R(G,A)$.
\end{lemma}
\begin{proof}
Write $\beta=\alpha+\lambda$. Then, $\check{\alpha}=\check{\beta}$. For any $a\in A$, one can show that
$$s_{\beta}s_{\alpha}(a)=a\check{\alpha}(\lambda(a)^{-1}).$$ By this $s_{\beta}s_{\alpha}(\alpha)=\alpha+
2\lambda$. As $s_{\beta}s_{\alpha}\in W_{tiny}$, we have $s_{\beta}s_{\alpha}(\lambda)=\lambda$. By these,
we get $\alpha+k\lambda\in R(G,A)$ for any $k\in\mathbb{Z}$.
\end{proof}

\begin{lemma}\label{L:strip7}
Let $\alpha\in R(G,A)-R_{0}(G,A)$ and $\lambda,\mu\in X^{\ast}_{0}(A)$. If $\alpha+\lambda\in R(G,A)$ and
$\alpha+\mu\in R(G,A)$, then $\alpha+\mu+2\lambda\in R(G,A)$; if $2\alpha+\lambda\in R(G,A)$ and $\alpha+
\mu\in R(G,A)$, then $\alpha+\mu+\lambda\in R(G,A)$ and $2\alpha+4\mu+\lambda\in R(G,A)$; if $2\alpha+
\lambda\in R(G,A)$ and $2\alpha+\mu\in R(G,A)$, then $2\alpha+\mu+2\lambda\in R(G,A)$.
\end{lemma}

\begin{proof}
The proof is similar to that of Lemmas \ref{L:strip5} and \ref{L:strip6}, all statements follow from
calculating the action of $s_{\alpha+\lambda}s_{\alpha}$ (or of $s_{2\alpha+\lambda}s_{\alpha}$,
$s_{\alpha+\mu}s_{\alpha}$) on $X^{\ast}(A)$.
\end{proof}


\begin{proposition}\label{P:strip8}
Let $\alpha\in R(G,A)-R_{0}(G,A)$. Then $kR_{0,\alpha}\subset R_{0,\alpha}$ for any $k\in\mathbb{Z}$,
and $$2R_{0,\alpha}+R_{0,\alpha}\subset R_{0,\alpha}.$$

Furthermore if $R_{2,\alpha}\neq\emptyset$, then $(2k+1)Z_{0,\alpha}\subset Z_{0,\alpha}$ for any $k\in
\mathbb{Z}$, \footnote{Since $0\in R_{0,\alpha}$, $Z_{0,\alpha}+R_{0,\alpha}\subset R_{0,\alpha}$ implies
that $Z_{0,\alpha}\subset R_{0,\alpha}.$}
$$Z_{0,\alpha}+R_{0,\alpha}\subset R_{0,\alpha},$$
$$4R_{0,\alpha}+Z_{0,\alpha}\subset Z_{0,\alpha},$$ $$2Z_{0,\alpha}+Z_{0,\alpha}\subset Z_{0,\alpha}.$$
\end{proposition}
\begin{proof}
Statements in the first part follow from Lemmas \ref{L:strip6} and \ref{L:strip7}.

Suppose $R_{2,\alpha}\neq\emptyset$. For any $\beta=2\alpha+\lambda\in R_{2,\alpha}$, by Lemma \ref{L:strip5}
we have $2\alpha+(2k+1)\lambda\in R(G,A)$ for any $k\in\mathbb{Z}$. Thus, $(2k+1)Z_{0,\alpha}\subset
Z_{0,\alpha}$. Other statements in the second part follow from Lemma \ref{L:strip7}.
\end{proof}

\smallskip

For each prime $p$, set $$R_{0,p}=\{\lambda\in X^{\ast}_{0}(A): [p^{k}\lambda]=1\textrm{ for some }k\geq 1\},$$
which is the $p$-torsion subgroup of $X^{\ast}_{0}(A)$.

\begin{theorem}\label{T:strip2}
Let $\alpha\in R(G,A)-R_{0}(G,A)$. Then there exists a subset $Y_{2}$ of $R_{0,2}$ and a subgroup $Y_{p}$
of $R_{0,p}$ for each prime $p>2$ such that $$R_{0,\alpha}=\big\{\sum_{p}\lambda_{p}:\lambda_{p}\in Y_{p}
\big\}.$$ The subset $Y_{2}$ satisfies $kY_{2}\subset Y_{2}$ for any $k\in\mathbb{Z}$, and $2Y_{2}+Y_{2}
\subset Y_{2}$.

Furthermore if $R_{2,\alpha}\neq\emptyset$, then there exists another subset $Z_{2}$ of $R_{0,2}$ such that
$$Z_{0,\alpha}=\big\{\sum_{p}\lambda_{p}:\lambda_{2}\in Z_{2},\ \lambda_{p}\in Y_{p} (\forall p>2)\big\}.$$
The subset $Z_{2}$ satisfies $(2k+1)Z_{2}\subset Z_{2}$ for any $k\in\mathbb{Z}$, $Z_{2}+Y_{2}\subset
Y_{2}$, $4Y_{2}+Z_{2}\subset Z_{2}$, $2Z_{2}+Z_{2}\subset Z_{2}$, and $2Y_{2}\cap Z_{2}=\emptyset.$
\end{theorem}

\begin{proof}
For any $\lambda=\sum_{p}\lambda_{p}\in R_{0,\alpha}$. Due to $kR_{0,\alpha}\subset R_{0,\alpha}$ for any
$k\in\mathbb{Z}$, we get $\lambda_{p}\in R_{0,\alpha}$ for any prime $p$. Write $Y_{p}$ for the image of
projection of $R_{0,\alpha}$ to $R_{0,p}$. Then, $$R_{0,\alpha}\subset\big\{\sum_{p}\lambda_{p}:\lambda_{p}
\in Y_{p}, \forall p\geq 2\big\},$$ and $Y_{p}\subset R_{0,\alpha}$ for any $p>2$. Using $2R_{0,\alpha}+
R_{0,\alpha}\subset R_{0,\alpha}$ and $Y_{p}\subset R_{0,\alpha}$ ($p>2$), we get $$R_{0,\alpha}=
\big\{\sum_{p}\lambda_{p}:\lambda_{p}\in Y_{p},\forall p\geq 2\big\}.$$

For any $\mu=\sum_{p}\mu_{p}\in Z_{0,\alpha}$. Due to $Z_{0,\alpha}\subset R_{0,\alpha}$, we have $\mu_{p}
\in Y_{p}$ for any prime $p$. Write $Z_{2}$ for the image of projection of $Z_{0,\alpha}$ to $R_{2,p}$.
Then, $$Z_{0,\alpha}\subset\big\{\sum_{p}\lambda_{p}:\lambda_{2}\in Z_{2},\ \lambda_{p}\in Y_{p}
(\forall p>2)\big\}.$$ Using $4R_{0,\alpha}+Z_{0,\alpha}\subset Z_{0,\alpha}$ and $Y_{p}\subset
R_{0,\alpha}$ ($p>2$), we get $$Z_{0,\alpha}=\big\{\sum_{p}\lambda_{p}:\lambda_{2}\in Z_{2},
\ \lambda_{p}\in Y_{p} (\forall p>2)\big\}.$$

The set $Y_{p}$ ($p>2$) is a group is due to $2R_{0,\alpha}+R_{0,\alpha}\subset R_{0,\alpha}$. The
conditions for $Y_{2},Z_{2}$ follow from that for $R_{0,\alpha},Z_{0,\alpha}$ in Proposition \ref{P:strip9}.

Finally, Lemma \ref{L:infinity3} indicates that $2R_{0,\alpha}\cap Z_{0,\alpha}=\emptyset$, which implies
that $2Y_{2}\cap Z_{2}=\emptyset.$
\end{proof}

\begin{question}\label{Q:Y-Z}
Whether the classification of $(Y_2,Z_2)$ using conditions in Theorem \ref{T:strip2} and the precise list
of all possible $(Y_2,Z_2)$ using the classification in \cite{Yu-maximalE} and the reduction in Subsection
\ref{SS:ast-semisimple} give the same list of pairs $(Y_{2},Z_{2})$.
\end{question}

\smallskip

In case $A/A^{0}$ is cyclic, the following proposition indicates that the strips are quite restrictive.
\begin{proposition}\label{P:strip9}
Assume $A/A^{0}$ is cyclic. For any $\alpha\in R(G,A)-R_{0}(G,A)$, $R_{0,\alpha}$ is a cyclic group.
Furthermore if $R_{2,\alpha}\neq\emptyset$, then there exists $\lambda\in X^{\ast}_{0}(A)$ of even degree
such that $$R_{0,\alpha}=\{k\lambda:k\in\mathbb{Z}\}\textrm{ and }Z_{0,\alpha}=\{(2k+1)\lambda:k\in
\mathbb{Z}\}.$$
\end{proposition}

\begin{proof}
For the first statement, by Theorem \ref{T:strip2} it suffices to show that $Y_{2}$ is a group. Choose
$\lambda\in Y_{2}$ of largest order. Then, $\mathbb{Z}\lambda\subset Y_{2}$ by Proposition
\ref{P:strip8}. As $A/A^{0}$ is cyclic, $\lambda$ is of largest order in $Y_{2}$ implies that $Y_{2}
\subset\mathbb{Z}\lambda$. Thus, $Y_{2}=\mathbb{Z}\lambda$ is a group.

For the second statement, again by Theorem \ref{T:strip2} it suffices to consider $Y_{2}$ and $Z_{2}$.
As $A/A^{0}$ is a cyclic group, so is $R_{0,2}$. By Proposition \ref{P:strip8}, we have $Z_{2}+Y_{2}
\subset Y_{2}$, $4Y_{2}+Z_{2}\subset Z_{2}$, $2Z_{2}+Z_{2}\subset Z_{2}$, and $2Y_{2}\cap Z_{2}=
\emptyset.$ These conditions imply that $Z_{2}=\{(2k+1)\lambda:k\in\mathbb{Z}\}$.
\end{proof}


\subsection{Relation between strips and FRD}

\begin{proposition}\label{P:strip-FRD1}
For each infinite root $\alpha$, we have $R_{0,\alpha}\subset R_{0}(G,A)\cup\{0\}$.
\end{proposition}

\begin{proof}
As $(\alpha'_{A^{0}},-\alpha|_{A^{0}})<0$ for any $\alpha'\in R_{1,\alpha}$, by Lemma \ref{L:strip1} we get
$$R_{0,\alpha}=R_{1,\alpha}+(-\alpha)\subset R_{0}(G,A)\sqcup\{0\}.$$
\end{proof}

By Proposition \ref{P:strip-FRD1}, we have $$|R_{0,\alpha}|\leq |R_{0}(G,A)|+1\leq |X^{\ast}(A/A^{0})|=
|A/A^{0}|.$$

\begin{theorem}\label{T:strip-FRD1}
Suppose $\frg=\frg_{\infty}$ and $R'(G,A)$ is an irreducible root system which is simply-laced, or is of
type $\C_{n}$ ($n\geq 3$), $\F_4$ or $\G_2$. Then $R_{0,\alpha}=R_{0}(G,A)\cup\{0\}$ for any short root
$\alpha\in R(G,A)-R_{0}(G,A)$.
\end{theorem}

\begin{proof}
Let $R'$ be the set of short roots in $R'(G,A)$. Under the assumption, $R'$ is a simply-laced irreducible
root system with rank $\geq 2$. By Theorem \ref{T:strip1}, $R_{0,\alpha}$ ($\alpha\in R(G,A)-R_{0}(G,A)$,
$\alpha|_{A^{0}}\in R'$) is constant, and it is a subgroup of $X^{\ast}(A/A^{0})$. Let it be denotes by
$R_{0}$. Since $\frg=\frg_{\infty}$, we have $$(\bigoplus_{\gamma\in R_{0}(G,A)}\frg_{\gamma})\oplus\frg_{0}
\subset\sum_{\alpha,\beta\in R(G,A)-R_{0}(G,A),\alpha|_{A^{0}}=\beta|_{A^{0}}}[\frg_{\alpha},\frg_{-\beta}].$$
When $R'(G,A)$ is simply-laced, the weight of any nonzero element in $[\frg_{\alpha},\frg_{-\beta}]$ is
$\alpha-\beta\in R_{0,\beta}=R_{0}$. Thus, $R_{0}(G,A)\subset R_{0}-\{0\}$. When $R'$ is not simply-laced,
if $\beta$ is a short root, then we also have $\alpha-\beta\in R_{0,\beta}\cup\{0\}=R_{0}\cup\{0\}$. If
$\beta$ is a long root, choose a short root $\alpha_1\in R(G,A)-R_{0}(G,A)$ with $(\alpha_1,\alpha)>0$.
Then, $\langle\alpha,\alpha_1\rangle=2$ and $\alpha_2=\alpha-\alpha_1\in R(G,A)-R_{0}(G,A)$. Set $\beta_1
=\alpha_1$. Similarly we have $\beta_2=\beta-\beta_1\in R(G,A)-R_{0}(G,A)$. All of $\alpha_1,\alpha_2,
\beta_1,\beta_2$ are short roots. Thus, $$\alpha-\beta=(\alpha_1+\alpha_2)-(\beta_1+\beta_2)=\alpha_2-
\beta_2\in R_{0,\beta_2}=R_{0}.$$ Therefore, $R_{0}(G,A)\subset R_{0}-\{0\}.$ By Proposition
\ref{P:strip-FRD1}, $R_{0}\subset R_{0}(G,A)\cup\{0\}$. Thus, $R_{0}=R_{0}(G,A)\cup\{0\}$.
\end{proof}

\begin{proposition}\label{P:strip-FRD2}
For an infinite root root $\alpha$ and a finite character $\lambda$ of order $n$, assume that $\alpha+\lambda$
is a root. Then, $\dim\frg_{[\lambda]}\geq n-1$ and $\frac{n}{\delta}$ divides $|R^{\vee}(\lambda)|,$ where
$\delta=2$ is $n$ is even and $\check{\alpha}(-1)=1$, and $\delta=1$ otherwise.
\end{proposition}

\begin{proof}
By Lemma \ref{L:strip6}, we have $R_{1,\alpha}\supset\{\alpha+k\lambda:0\leq k\leq n-1\}$. By Proposition
\ref{P:strip-FRD1}, we have $k\lambda\in R_{0}(G,A)$ ($0\leq k\leq n-1$). Thus, $\dim\frg_{[\lambda]}\geq
n-1.$ By the proof of Theorem \ref{T:W0}, we have $s_{\alpha}s_{\alpha+\lambda}
\in R^{\vee}(\lambda)$. Note that $$s_{\alpha}s_{\alpha+\lambda}(a)=a\check{\alpha}(\lambda(a)),\ \forall a
\in A.$$ Thus, $\frac{n}{\delta}||R^{\vee}(\lambda)|$.
\end{proof}

When $\beta=2\alpha+\lambda$ is a root, $n$ is even and $\check{\alpha}(-1)=1$. In this case $$s_{\alpha}
s_{\alpha+\lambda}(a)=a\check{\beta}(\lambda(a)), \forall a\in A.$$ Thus, $n|R^{\vee}(\lambda)$.

\smallskip

We propose the following assumption.

\textbf{$R_{0}$-assumption}: $\frg=\frg_{\infty}$, $R'(G,A)$ is an irreducible root system, and $R_{0,\alpha}$
for short infinite roots $\alpha$ in $R(G,A)$ is constant, and is a subgroup of $X^{\ast}(A/A^{0})$, denoted
by $R_0$.

\begin{proposition}\label{P:strip-FRD3}
Under the $R_{0}$-assumption, $R_{0}(G,A)=R_{0}-\{0\}$, and $$\{a\in A:\lambda(a)=1,\forall\lambda\in R_{0}\}
=A^{0}.$$
\end{proposition}

\begin{proof}
From the proof of Theorem \ref{T:strip-FRD1}, one can show that $R_{0}(G,A)=R_{0}-\{0\}$.

Choose a simple system $\{\alpha'_{i}:1\leq i\leq l\}$ of $R'(G,A)$. For each $i$, choose $\alpha_{i}\in
R(G,A)$ with $\alpha_{i}|_{A^{0}}=\alpha'_{i}$. Let $R'$ be the sub-root system generated by $\alpha_1,
\dots,\alpha_{l}$. Then, $\alpha\mapsto\alpha|_{A^{0}}$ ($\alpha\in R'$) gives an isomorphism from $R'$
to $R'(G,A)$. As we assume $G$ is a of adjoint type. Thus, $$\{a\in A:\lambda(a)=1,\forall\lambda\in
R(G,A)\}=1.$$ Hence, $R(G,A)$ generates $X^{\ast}(A)$. By the above, $R'$ and $R_0$ generate $X^{\ast}(A)$.
Suppose $\{a\in A:\lambda(a)=1,\forall\lambda\in R_{0}\}\neq A^{0}.$ Then, there exists $a\in A-A^{0}$
such that $aA^{0}\subset\{a\in A:\lambda(a)=1,\forall\lambda\in R_{0}\}$. Furthermore, there exists
$a'\in aA^{0}$ such that $\alpha(a')=1$ for any $\alpha\in R'$. As $a\not\in A^{0}$, we have $a'\neq 1$.
Thus, $$\{a\in A:\lambda(a)=1,\forall\lambda\in R_{0}\}=A^{0}.$$
\end{proof}

Under the $R_{0}$-assumption, for each $\lambda\in R_{0}-\{0\}=R_{0}(G,A)$, set $$R^{\vee}_{0}(\lambda)=
\langle\{s_{\alpha}s_{\alpha+\lambda}:\alpha\in R'\}\rangle.$$ By Theorem \ref{T:W0}, $R^{\vee}_{0}(\lambda)
\subset R^{\vee}(\lambda)$. Write $W'_0$ for the subgroup of $W_{small}(G,A)$ generated by $R^{\vee}_{0}
(\lambda)$. Then, $$W'_{0}\subset W'\subset\Hom(A/A^{0},A^{0}),$$ and $$W_{tiny}=W'_{0}\rtimes W_{R'}.$$
There is a natural bi-additive map $$\phi:X^{\ast}_{0}(A)\times X_{\ast}(A^{0})\rightarrow\Hom(A/A^{0},
A^{0}).$$ Write $L'=\span_{\mathbb{Z}}\{\check{\alpha'}:\alpha'\in R'(G,A)\}$. Using the formula
$$s_{\alpha}s_{\alpha+\lambda}(a)=a\check{\alpha}(\lambda(a))(\forall a\in A),$$ we see that $W'_0$ could
be identified with $\phi(X^{\ast}_{0}(A)\otimes L')\subset\Hom(A/A^{0},A^{0})$. In general $W_{small}\neq
W_{1}$, which forces $W_{small}\neq W_{tiny}$. Hence, there may exist $\lambda\in R_{0}-\{0\}$ such that
$R^{\vee}_{0}(\lambda)\neq R^{\vee}(\lambda)$.

\begin{question}\label{Q:strip-FRD1}
Under the $R_{0}$-assumption, is $W'=W'_{0}$?
\end{question}

\begin{question}\label{Q:strip-FRD2}
Under the $R_{0}$-assumption, can one determine the twisted root system $(R(G,A),R^{\vee})$ and the Lie
algebra $\frg$ from the root system $R'(G,A)$ and the group $R_0$? Precisely, this consists of determining
$\dim\frg_{\lambda}$ ($\lambda\in R_{0}-\{0\}$), $R^{\vee}(\lambda)$ ($\lambda\in R_{0}-\{0\}$), $W_{small}$,
and Lie bracket relations between root spaces.
\end{question}

With the classification of $(\ast)$-subgroups in \cite{Yu-maximalE}, we have a precise list of $(\ast)$-subgroups
satisfying the $R_{0}$-assumption (cf. Table 1). By that we can verify Question \ref{Q:strip-FRD1}. However we
prefer a Lie theoretical proof. To do this, besides tools developed in this paper, we might need to discover more
structures of twisted root system associated to $(\ast)$-subgroups.

Assume $\frg$ is simple. From Table 1, we see that there is only one example where $A$ satisfies the
$R_0$-assumption but is not a maximal abelian subgroup. That happens for the $(\ast)$-subgroup $F_{18}$ in
$\Aut(\fre_7)$ as defined in \cite[Page259, Eq.(115)]{Yu-maximalE}.

\begin{question}\label{Q:Weyl}
Suppose $\frg_0$ is a compact simple Lie algebra and $A$ is a maximal abelian subgroup of $\Aut(\frg_0)$
with positive dimension. Is $W_{large}(G,A)=W_{small}(G,A)$?
\end{question}

We can verify Question \ref{Q:Weyl} with the classification of maximal abelian subgroups and the calculation
of their large Weyl groups in \cite{Yu-maximalE}. However, a Lie theoretical study is preferred.

\subsection{A List of the root systems $R'(G,A)$}

Now assume $\fru_0$ is a compact simple Lie algebra, and $A$ is a $(\ast)$-subgroup of $\Aut(\fru_0)$ of positive
dimension. We show how to {\it calculate the root system $R'(\Aut(\fru_0),A)$}.

When $A\subset\Int(\fru_0)$ ({\it the inner case}), write $G=\Int(\fru_0)$. By \cite[Lemma 2.3]{Yu-maximalE} we
have $A^{0}=Z(L)^{0}$, where $L=Z_{G}(A^{0})$ is a Levi subgroup of $G$. We may assume that the root system of
$\frl=Z_{\frg}(A^{0})$ has a simple system $\Pi'$ contained in a simple system $\Pi$ of $\frg$. Write $\Lambda,
\Lambda'$ for the root lattices of $\frg,\frl$ respectively. Then, we could identify the dual space of $\fra_0$
with the orthogonal complement of $V'=\Lambda'\otimes_{\mathbb{Z}}\mathbb{R}$ in $V=\Lambda\otimes_{\mathbb{Z}}
\mathbb{R}$, denoted by $V''$. Let $\phi: V\rightarrow V''$ be the projection map. Then, $\Pi''=\phi(\Pi-\Pi'')$
is a simple system of $R'(G,A)$. By Theorem \ref{T:R'-simple}, $R'(G,A)$ is an irreducible root system. By
calculating cusp products of elements in $\Pi''$ and checking whether the double of a short element in $\Pi''$
is contained in $R'(G,A)$ (this is used in distinguishing type $\BC_{n}$ from type $\B_{n}$), we identified
the root system $R'(G,A)$.

When $A\not\subset\Int(\fru_0)$ ({\it the outer case}), write $G=\Aut(\fru_0)$. By \cite[Page 247,Lemma 5.1]
{Yu-maximalE} and its analogue for any pair $(\fru_0,\tau)$ with $\fru_0$ a compact simple Lie algebra and
$\tau$ a pinned automorphism\footnotemark,\footnotetext{$\tau$ is called a pinned automorphism of $\fru_0$
if the induced automorphism of $\tau$ on $\frg$ stabilizes a tuple $(\frt,\Delta^{+},\{X_{\alpha}:\alpha\in
\Delta^{+}\})$. Particularly for each $\alpha\in\Delta^{+}$, $\tau(X_{\alpha})=X_{\alpha}$ whenever
$\tau(\alpha)=\alpha$. Here $\frt$ is a Cartan subalgebra of $\frg$, $\Delta^{+}$ is a positive system of
the root system $\Delta(\frg,\frt)$, and $X_{\alpha}$ is a root vector for the root $\alpha$. The proof of
\cite[Page 247, Lemma 5.1]{Yu-maximalE} is valid for any pinned automorphism.}
$\fra$ is conjugate to the center of a Levi subalgebra of the symmetric subalgebra $\frk=\frg^{\tau}$. We may
assume that $\fra$ is equal to the center of a Levi subalgebra $\frl'$ of $\frk=\frg^{\tau}$, and the root
system of $\frl'$ has a simple system $\Pi'$ contained in a simple system $\Pi$ of $\frk$. Then, a similar
way of calculation by studying the projection of $\Pi''=\Pi-\Pi'$ to the orthogonal complement of
$\mathbb{R}\cdot\Pi'$ in $\mathbb{R}\cdot\Pi$ as in the inner case applies to identify $R'(G,A)$.

\smallskip

The following lemma gives a way of calculating $|R_{0}(G,A)|$ and $|R_{1,\alpha}|$, $|R_{2,\alpha}|$ for any
$\alpha\in R(G,A)-R_{0}(G,A)$.

\begin{lemma}\label{L:alpha-dim}
Let $A$ be a $(\ast)$-subgroup of a compact Lie group $G$. Then $$\dim Z_{\frg}(A^{0})=\dim A+\sum_{\alpha
\in R_{0}(G,A)}\dim\frg_{\alpha}.$$ For any infinite root $\alpha$, we have $$\dim Z_{\frg}
(\ker\alpha|_{A^{0}})-\dim Z_{\frg}(A^{0})=2|R_{\alpha}|,$$ $$\dim Z_{\frg}(\ker 2\alpha|_{A^{0}})-
\dim Z_{\frg}(A^{0})=2|R_{2,\alpha}|.$$
\end{lemma}

\begin{proof}
Each of $\dim Z_{\frg}(A^{0})$, $Z_{\frg}(\ker\alpha|_{A^{0}})$ and $Z_{\frg}(2\ker\alpha|_{A^{0}})$ is a
direct sum of $\fra$ and some root spaces. For a root $\beta\in R(G,A)$, $\frg_{\beta}\subset Z_{\frg}(A^{0})$
if and only if $\beta\in R_{0}(G,A)$; $\frg_{\beta}\subset Z_{\frg}(\ker\alpha|_{A^{0}})$ if and only if
$\beta\in R_{0}(G,A)\cup \pm{R_{\alpha}}$; $\frg_{\beta}\subset Z_{\frg}(\ker2\alpha|_{A^{0}})$ if and
only if $\beta\in R_{0}(G,A)\cup\pm R_{2,\alpha}$. From this, we get $$\dim Z_{\frg}(A^{0})=\dim A+
\sum_{\alpha\in R_{0}(G,A)}\dim\frg_{\alpha},$$
$$\dim Z_{\frg}(\ker\alpha|_{A^{0}})-\dim Z_{\frg}(A^{0})=2|R_{\alpha}|,$$
$$\dim Z_{\frg}(\ker 2\alpha|_{A^{0}})-\dim Z_{\frg}(A^{0})=2|R_{2,\alpha}|.$$
\end{proof}


\begin{proposition}\label{P:strip-FRD4}
Suppose $\frg=\frg_{\infty}$ and $R'(G,A)$ is of type $\B_{n}$ ($n\geq 3$) or $\BC_{n}$ ($n\geq 3$),
then $R_{0,\alpha}$ for long roots (in $\B_{n}$ case) or middle length roots (in $\BC_{n}$ case)
is constant and it is a subgroup of $X^{\ast}(A/A^{0})$.
\end{proposition}
\begin{proof}
Under the assumption, long roots in the $\B_{n}$ ($n\geq 3$) case or middle length roots in the $\BC_{n}$
($n\geq 3$) case is a root system of type $D_{n}$, which is simply-laced and of rank $\geq 3$. The
conclusion follows from Theorem \ref{T:strip1}.
\end{proof}

Let $R'_{0}$ denote the constant group $R_{0,\alpha}$ in Proposition \ref{P:strip-FRD4}.

\begin{proposition}\label{P:strip-FRD5}
In the $\B_{n}$ ($n\geq 3$) case, for any short root $\beta$, $R_{0,\beta}\supset R'_{0}$. In the $\BC_{n}$
($n\geq 3$) case, for any short root $\beta$, $R_{0,\beta}\supset R'_{0}$; for any long root $\beta$,
$R_{0,\beta}\subset R'_{0}$.
\end{proposition}
\begin{proof}
We show the statement for short roots in the $\B_{n}$ ($n\geq 3$) case. The proof of the other two
statements is similar.

Suppose $R'(G,A)$ is of type $\B_{n}$ ($n\geq 3$). Let $\beta$ be a short infinite root. Choose a long
infinite root $\alpha$ with $(\alpha|_{A^{0}},\beta|_{A^{0}})>0$. Then, $\gamma=\alpha-\beta\in R(G,A)$
by Lemma \ref{L:strip1}. It is clear that $\gamma$ is a short infinite root. For any $\alpha'\in
R_{1,\alpha}$, due to $(\alpha'|_{A^{0}},\gamma|_{A^{0}})>0$, we have $\alpha'-\gamma\in R(G,A)$. It is
clear that $\alpha'-\gamma\in R_{1,\beta}$. Thus, $$R_{0,\alpha}=R_{1,\alpha}-\alpha=(R_{1,\alpha}-
\gamma)-\beta\subset R_{1,\beta}-\beta=R_{0,\beta}.$$ This shows $R_{0,\beta}\supset R'_{0}$.
\end{proof}

By Table 1 below, when $\frg_0$ is a compact simple Lie algebra, $R'(G,A)$ is of type $\B_{n}$ ($n\geq 3$)
or $\BC_{n}$ ($n\geq 3$) only happens when $A$ is a $(\ast)$-subgroup of $G=\PO(n)$, $\PSp(n)$ or
$\Aut(\mathfrak{su}(n))$. In the last case, it requires $A\not\subset G^{0}$. In this case, in
\cite{Yu-maximalE}, it is associated with a bi-multiplicative function with values in $\{\pm{1}\}$ on $A$
(in the $\PO(n)$ and $\PSp(n)$ case) or $A\cap G^{0}$ (in the $\Aut(\mathfrak{su}(n))$ case).

\begin{proposition}\label{P:strip-FRD6}
Suppose $\frg$ is simple and $R'(G,A)$ is of type $\B_{n}$ ($n\geq 3$) or $\BC_{n}$ ($n\geq 3$).
Then, $$\{a\in A:\lambda(a)=1,\forall\lambda\in R'_{0}\}=\ker m.$$
\end{proposition}
\begin{proof}
This follows from the precise form of the subgroup $A$ given in \cite{Yu-maximalE}.
\end{proof}

In Table 1, we describe the root system $R'(G,A)$ for every $(\ast)$-subgroup $A$ in $\Aut(\fru_0)$ with
$\fru_0$ a compact simple Lie algebra and $0<\dim A<\rank\fru_0$. We also give the multiplicity of any
root $\alpha'$ in $R'(G,A)$, which depends only on the length of $\alpha'$.

The label ``Y" in the last three columns means $R_{0,\alpha}$ ($\alpha\in R(G,A)$, $\alpha|_{A^{0}}=
\alpha'$) is a group; without labeling ``Y" means it is not a group except for few cases, e.g., for
short roots of $(\ast)$-subgroups in the classical simple Lie algebra case where the integer $s_0=1$.
Note that $R_{0,\alpha}$ depends only on the $W_{tiny}$ orbit $W_{tiny}\cdot\alpha$; and $R_{0,\alpha}$
is a group if and only if $R_{0,\alpha'}=R_{0,\alpha}$ for any $\alpha'\in R_{1,\alpha}$.

The label ``R" (=regular) in the third column means the ``$R_0$-assumption" is satisfied, i.e., $R_{0,\alpha}$
for short infinite roots is a constant subgroup of $X^{\ast}(A/A^{0})$; the label ``HR"(=half-regular) means
$R'(G,A)$ is of type $\B_{n}$ ($n\geq 2$) and $R_{0,\alpha}$ for long infinite roots is a constant subgroup
of $X^{\ast}(A/A^{0})$, or $R'(G,A)$ is of type $\BC_{n}$ ($n\geq 2$) and $R_{0,\alpha}$ for middle length
infinite roots is a constant subgroup of $X^{\ast}(A/A^{0})$; the label ``S"(=singular) means neither of
the above two holds.

The notation $F_{1},\dots,$ and the invariants $(m,n_{1},\dots,n_{s})$, $(k,s_0,s_1)$ are from \cite{Yu-maximalE}.
In Table 1, there is one example of $(\ast)$-subgroup omitted (we is also not discussed in \cite{Yu-maximalE}),
that is a $(\ast)$-subgroup $A$ of $G=\Aut(\mathfrak{so}(8))$ of the form $A=A^{0}\times\langle\theta\rangle$
with $\theta$ a pinned automorphism of order $3$ and $A^{0}$ a maximal torus of $(G^{0})^{\theta}\cong\G_2$.
In this case, $R'(G,A)$ is of type $\G_2$; $|R_{0,\alpha}|=3$ for short infinite roos, and $|R_{0,\alpha}|=1$
for long infinite roots. Then, it satisfies the $R_{0}$-assumption.



\begin{table}[ht]
\caption{Root system $R'(G,A)$} \centering
\begin{tabular}{|c |c |c |c |c |c|c|}\hline
$G$ & $A$  & $R'(G,A)$ & short & long & med \\ [0.3ex] \hline
$\PU(n)$& $\frac{n}{m}\!=$ & $\A_{m-1}$,R & $(\frac{n}{m})^{2}$,\!Y &&\\
($n\geq 2$)&$\!n_{1}\!\cdots\!n_{s}$&&&&\\\hline
$\PO(n)$\footnotemark & $\frac{n}{2^{k}}\!=\!$ & $\D_{s_1}$($k\!=\!s_0\!=\!0$),R & $2^{2k}$,\!Y &&\\
($n\geq 5$)& $s_0\!+\!2s_1$ &$\C_{s_1}$\!($k\!>\!0,\!s_0\!=\!0$\!),\!R&$2^{2k}$,\!Y&$2^{k\!-\!1}\!(2^{k}\!-\!1\!)$& \\
 &  &$\!\B_{s_1}$\!($\!k\!=0,\!s_0\!>\!0\!$\!),\!HR\!& $s_0$ & $1$,Y & \\
 &  & $\BC_{s_1}$($ks_{0}\!>\!0$),HR &$\!s_{0}\!\cdot\!2^{2k}\!$&$2^{k\!-\!1}\!(2^{k}\!-\!1\!)$& $2^{2k}$,\!Y\\\hline
$\PSp(n)$& $\frac{n}{2^{k}}\!=\!$ & $\C_{s_1}$($s_0\!=\!0$),R &$2^{2k}$,Y &$2^{k\!-\!1}\!(2^{k}\!+\!1\!)$&\\
($n\geq 2$)& $s_0\!+\!2s_1$ & $\BC_{s_1}$($s_{0}\!>\!0$),HR &$\!s_{0}\!\cdot\!2^{2k}\!$&$2^{k\!-\!1}\!(2^{k}\!+\!1\!)$& $2^{2k}$,\!Y\\\hline
$\!\Aut\!(\mathfrak{su}(n\!)\!)\!$&$\frac{n}{2^{k}}\!=\!$ &$\C_{s_1}$($s_0\!=\!0$),R&$\!2^{2k\!+\!1}\!$,\!Y\!& $\!2^{2k}$,\!Y\!&\\
($n\geq 3$)& $s_0\!+\!2s_1$ & $\BC_{s_1}$($s_{0}\!>\!0$),HR &$\!s_{0}\!\cdot\!2^{2k}\!$&$2^{2k}$,Y &$\!2^{2k\!+\!1}\!$,\!Y\!\\\hline
$\Aut(\fre_6)$ & $F_4$ & $\G_2$,R & $9$,Y & $1$,Y &\\
& $F_5$ & $\A_2$,R & $8$,Y & &\\
& $F_{13}$ & $\BC_1$,R & $16$,Y & $8$,Y &\\
& $F_{14}$ & $\BC_2$,R & $8$ & $1$,Y & $6$ \\
& $F_{15}$ & $\BC_1$,S & $20$ & $1$,Y &\\
& $F_{16}$ & $\F_4$,R & $2$,Y & $1$,Y &\\\hline
$\Aut(\fre_7)$ & $F_{13}$ & $\F_4$,R & $4$,Y & $1$,Y &\\
& $F_{14}$ & $\C_3$,R & $8$,Y & $1$,Y &\\
& $F_{15}$ & $\BC_2$,R & $16$,Y & $1$,Y & $8$\\
& $F_{16}$ & $\BC_1$,R & $32$,Y & $10$ & $6$ \\
&$\!F_{\!17\!}\!-\!F_{\!19\!}\!$ & $\BC_1$,R or S\footnotemark & $32$,Y & $1$,Y &\\
& $F_{20}$ & $\A_1$,R & $27$,Y & &\\\hline
$\Aut(\fre_8)$ & $F_{14}$ & $\F_4$,R & $8$,Y & $1$,Y &\\
& $F_{15}$ & $\BC_2$,R & $32$,Y & $1$,Y& $12$ \\
& $F_{16}$ & $\BC_1$,R & $64$,Y & $14$  &\\
& $F_{17}$ & $\G_2$,R & $27$,Y & $1$,Y &\\
& $F_{18},\!F_{19}$ & $\BC_1$,S & $56$ & $1$,Y &\\\hline
$\Aut(\frf_4)$ & $F_{3}$ & $\BC_1$,R & $8$,Y & $7$ &\\\hline
\end{tabular}
\end{table}

\footnotetext{In the $\PO(n)$ case there is an exception when $s_1=1$ and $s_{0}>0$. In this case the two
''HR" should be changed to ``S". Precisely, when $k=0$ and $s_0>0$, $R'(G,A)$ is of type $\B_1$, and there
are no long roots; when $ks_0>0$, $R'(G,A)$ is of type $\BC_1$, and there are no middle length roots.
In the $\E_7$ case, $F_{17},F_{18}$ are regular, but $F_{19}$ is singular.}

\section{Other subjects}\label{S:openQ}

Besides subjects discussed in the main body of the paper, we discuss some other subjects in this final section.
We also pose a few questions on these subjects.

\subsection{Isogenous subgroups}\label{SS:isogeny}

Let $A_1$ (or $A_2$) be a $(\ast)$-subgroup of a compact Lie group $G_1$ (or $G_2$). We say the pairs
$(A_1,G_2)$ and $(A_2,G_2)$ have {\it basic isogeny relation} if one of the following or their converse
by switching $(A_1,G_2)$ and $(A_2,G_2)$ holds, \begin{itemize}
\item[1,] there is a surjective isogeny $\phi: G_1\rightarrow G_2$ such that $\phi(A_1)=A_2$.
\item[2,] there is an injective isogeny $\phi: G_1\hookrightarrow G_2$ such that $\phi(A_1)=A_2$.
\item[3,] there is an isomorphism $\phi: G_1\rightarrow G_2$ such that $\phi(A_1)\subset A_2$.
\end{itemize}

\begin{definition}\label{D:isogeny2}
Let $A$ (or $A'$) be a $(\ast)$-subgroup of $G$ (or $G'$). We say the pair $(A,G)$ is isogenous to
$(A',G')$ if there is a chain of $(\ast)$-subgroups $(A_{i},G_{i})$ ($0\leq i\leq s$) such that $(A_1,G_1)
=(A,G)$, $(A_{s},G_{s})=(A',G')$, and for each $i$ ($1\leq i\leq s$), $(A_{i-1},G_{i-1})$ and
$(A_{i},G_{i})$ have basic isogeny relation.
\end{definition}

\begin{lemma}\label{L:isogeny1}
Suppose $\phi: G_1\rightarrow G_2$ gives a basic isogeny relation between a $(\ast)$-subgroup $A_1$ of
$G_1$ and a $(\ast)$-subgroup $A_2$ of $G_2$. Then, the induced map $$\phi^{\ast}: R(G_{2},A_{2})\!-
\!R_{0}(G_{2},A_{2})\longrightarrow R(G_{1},A_{1})\!-\!R_{0}(G_{1},A_{1})$$ is a bijection.
\end{lemma}
\begin{proof}
For surjective isogeny and injective isogeny, this statement is clear. Now assume that $\phi$ is an
isomorphism and $\phi(A_1)\subset A_2$. Since infinite root spaces generate
$\frg_{i}/Z_{\frg_{i}}(A_{i}^{0})$ and each infinite root space has dimension one by Lemma
\ref{L:infinity1}, it follows that $$\phi^{\ast}:R(G_{2},A_{2})\!-\!R_{0}(G_{2},A_{2})\longrightarrow
R(G_{1},A_{1})\!-\!R_{0}(G_{1},A_{1})$$ is a surjection. As $$\dim\frg_{1}/Z_{\frg_{1}}(A_{1}^{0})=
\dim\frg_{2}/Z_{\frg_{2}}(A_{2}^{0}),$$ we have $|R(G_{2},A_{2})\!-\!R_{0}(G_{2},A_{2})|=
|R(G_{1},A_{1})\!-\!R_{0}(G_{1},A_{1})|.$ Thus, $\phi^{\ast}$ is a bijection.
\end{proof}

Roughly to say, Lemma \ref{L:isogeny1} means ``isogenous $(\ast)$-subgroups have same infinite roots".
This is stronger than the following statement.

\begin{lemma}\label{L:isogeny2}
If $(A_1,G_2)$ and $(A_2,G_2)$ are isogenous, then there is an isomorphism $\phi: A_{1}^{0}\rightarrow
A_{2}^{0}$ such that $$\phi^{\ast}(R'(G_2,A_2))=R'(G_1,A_1).$$
\end{lemma}

\begin{lemma}\label{L:isogeny3}
Suppose $\frg=\frg_{\infty}$. Then, $$R_{0}(G,A)\sqcup\{0\}=\bigcup_{\alpha\in R(G,A)-R_{0}(G,A)}
R_{0,\alpha}.$$
\end{lemma}

\begin{proof}
Since $\frg=\frg_{\infty}$, we have $$\fra\oplus\fra_{0}\bigoplus_{\lambda\in R_{0}(G,A)}\frg_{\lambda}=
\sum_{\alpha,\beta\in R(G,A)-R_{0}(G,A),\alpha|_{A^{0}}=-\beta|_{A^{0}}}[\frg_{\alpha},\frg_{\beta}].$$
This just means, $$R_{0}(G,A)\sqcup\{0\}=\bigcup_{\alpha\in R(G,A)-R_{0}(G,A)}R_{0,\alpha}.$$
\end{proof}

Lemma \ref{L:isogeny3} gives some hope to compare finite roots and twisted root system of isogenous
$(\ast)$-subgroups. When $G$ is semisimple, $(A,G)$ is isogenous to $(\pi(A),\Aut(\frg_0))$, where
$\pi: G\rightarrow\Aut(\frg_0)$ is the adjoint homomorphism. Then, we could stick to $(\ast)$-subgroups
to $\Aut(\frg_0)$, and the consideration becomes a bit easier. 

When $\frg_0$ is simple, the following statement indicates that in every class of isogenous subgroups
of $\Aut(\frg_0)$ with positive dimension there is a unique maximal abelian subgroup. It is a pity
that we can verify Proposition \ref{P:isogeny4} only by using the classification in \cite{Yu-maximalE},
and we don't know how to formulate a corresponding statement for isogenous subgroups in general compact
Lie groups. We have't checked if Proposition \ref{P:isogeny4} holds for finite $(\ast)$-subgroups.

\begin{proposition}\label{P:isogeny4}
Let $\frg_0$ be compact simple Lie algebra, and $A\subset\Aut(\frg_0)$ be a positive-dimensional $(\ast)$-subgroup.
Then there is a unique conjugacy class of maximal abelian subgroups in the isogeny class of $A$.
\end{proposition}

\begin{proof}
When $A$ is a $(\ast)$-subgroup of $\PU(n)$, $A$ itself is a maximal abelian subgroup. In the $\PO(n)$, $\PSp(n)$,
$\Aut(\mathfrak{su}()n)$ case, by Table 1, we see that infinite roots of $A$ determine the invariants $k,s_{0},s_{1}$.
Thus, the conclusion follows. When $\frg_0$ is an exceptional simple Lie algebra, $A$ itself is maximal except when
$\frg_0$ is of type $\E_7$ and $A$ is conjugate to the subgroup $F_{18}$ defined in
\cite[Page 259, Eq.(115)]{Yu-maximalE}. In this case, $F_{19}$ represents the unique conjugacy class of maximal
abelian subgroups in the isogenous class of $F_{18}$.
\end{proof}

Isogeny classes of finite $(\ast)$-subgroups are more complicated.  A particular question is as follows.
\begin{question}\label{Q:isogeny1}
Let $G$ be a compact Lie group, and $A_1,A_2$ be finite $(\ast)$-subgroups which are isogenous. Are the prime
decompositions associated to them the same.
\end{question}

Write $\frg_{1,p}$ (or $\frg_{2,p}$) be the sum of root spaces of $A_1$ (or $A_2$) of $p$-power order roots.
We say the prime decompositions are the same if there is $\phi\in\Aut(\frg)$ such that $\phi(\frg_{1,p})=
\frg_{2,p}$ for any prime $p$.

\subsection{Good $(\ast)$-subgroups}\label{SS:good}

Let $A$ be a $(\ast)$-subgroup of a compact Lie group $G$. We call $A$ a {\it good $(\ast)$-subgroup} if the
following conditions hold, \begin{enumerate}
\item for any finite character $\lambda$ of $A$, if there exists $1\neq\xi\in\ker\lambda$ such that
$s_{\lambda,\xi}\in W_{large}(G,A)$, then $\frg_{[\lambda]}\neq 0$.
\item for any generalized finite root $\lambda$, $\frg_{\lambda}\neq 0$, i.e., $\lambda$ is a finite root.
\item for any finite root $\alpha$, \[R^{\vee}(\alpha)=\{\xi\in\ker\alpha: s_{\alpha,\xi}\in W_{small}(G,A)\}.\]
\end{enumerate}
If $A$ is a good $(\ast)$-subgroup, then $W_{middle}(G,A)=W_{small}(G,A)$, all generalized finite roots are finite
roots, and each $R^{\vee}(\alpha)$ (for $\alpha$ a root) is determined by $W_{small}(G,A)$.

The conditions posed above for ``good abelian subgroup" are natural. When $G=\Aut(\frg_0)$ with $\frg_0$ a compact
simple Lie algebra, we verified most maximal abelian subgroups and some non-maximal ones satisfy these conditions.
But we haven't checked all cases, particularly we are not certain if all maximal abelian subgroups are good.

\subsection{Semisimplicity}\label{SS:semisimple}

\begin{definition}\label{D:semisimple1}
Let $A$ be a $(\ast)$-subgroup of a compact Lie group $G$. We say $A$ (or the pair $(A,G)$) semisimple if
$$\{a\in A:\alpha(a)=1,\forall\alpha\in R(G,A)\}$$ is a finite group. We say $A$ (or the pair $(A,G)$) of
adjoint type if $$\{a\in A:\alpha(a)=1,\forall\alpha\in R(G,A)\}=1.$$
\end{definition}

If $G$ itself is a semisimple Lie group, then any $(\ast)$-subgroup of it is semisimple, but not vice versa.
For example, any finite $(\ast)$-subgroup is semisimple. If $G$ itself is of adjoint type, then any
$(\ast)$-subgroup of it is semisimple, again not vice versa. For example, $\E_{6}^{sc}$ has a $(\ast)$-subgroup
isomorphic to $(C_{3})^{3}$ (cf. the group $F_2$ in \cite[Page 243, Table 4]{Yu-maximalE}) which is of adjoint
type.

Note that a root system $\Phi$ in a lattice $L$ is said to be semisimple if $\rank\Phi=\rank L$, i.e., if
$\Phi$ spans a sublattice of $L$ with finite index.

\begin{lemma}\label{L:semisimple2}
Let $A$ be a $(\ast)$-subgroup of a compact Lie group $G$. Then $A$ is semisimple if and only if $R'(G,A)
=\emptyset$ or $R'(G,A)$ is a semisimple root system in the lattice $X^{\ast}(A^{0})$.
\end{lemma}
\begin{proof}
Just note that $$\{a\in A:\alpha(a)=1,\forall\alpha\in R(G,A)\}$$ is a finite group if and only if $R'(G,A)
=\emptyset$ or $R'(G,A)$ is a semisimple root system in the lattice $X^{\ast}(A^{0})$.
\end{proof}

\subsection{The action of $W_{f}$ on $R(G,A)$}\label{SS:Wf}

The following example indicates that the action of $W_{f}$ on $R(G,A)$ and root spaces gives interesting examples
of group action. We might wish that it also reveal more structures of $R(G,A)$.

\begin{example}\label{E:E8}
Let $G=\Aut(\fre_8)$, and $A=F_1$ be as defined in \cite[Page 261, Eq. (117)]{Yu-maximalE}. Then, $A\cong
(\mathbb{Z}/5\mathbb{Z})^{3}$ and it is a maximal abelian subgroup. By the proof of
\cite[Proposition 7.1]{Yu-maximalE}, we have $$W_{small}(G,A)\cong\SL(2,\mathbb{F}_{5}).$$ Write
$W=W_{small}(G,A)$. Take any $\alpha\in R(G,A)$. Then, $\Stab_{W}(\alpha)\cong\SL(2,\mathbb{F}_{5}).$ From
these, we see that $\dim\frg_{\alpha}=2$ and the action of $\Stab_{W}(\alpha)$ on $\frg_{\alpha}$ gives
the binary icosahedral subgroup of $\SL(2,\mathbb{C})$.
\end{example}

By Theorems \ref{T:W0}, Theorem \ref{T:strip-FRD1}, etc, we exhibited finite roots and coroot groups from
infinite roots. On the converse direction, one may study the action of $W_{f}$ on infinite roots to show
restriction of infinite roots from finite root datum. However, the formula for the action of a transvection
$s_{\alpha,\xi}$ on $X^{\ast}(A^{0})$ looks complicated, which causes difficulty to fulfilling this study.

\subsection{Singular root systems}\label{SS:singular}

The case when $R'(G,A)$ is of type $\A_{1}$, $\BC_{1}$, $\B_2$ or $\BC_{2}$ is most singular. In this case
playing reflections gives few informaiton for the structure of strips and finite root datum. In order to
understand general twisted root systems, we need to study further in case $R'(G,A)$ is one of the types:
$\A_{1}$, $\BC_{1}$, $\B_2$, $\BC_{2}$.

\subsection{Root system in a lattice}\label{SS:rootSystem}

Let $G$ be a compact Lie group with a biinvariant Riemannian metric, and $\theta$ be an involutive automorphism
of $G$. We call a closed abelian subgroup $A$ of $G$ a {\it symmetric $(\ast)$-subgroup\footnotemark} if
$\theta(a)=a^{-1}$ for any $a\in A$, and $\fra_0$ is a maximal abelian subspace of $$(\frg_0^{A})^{-\theta}=
\{X\in\frg_0:\Ad(g)X=X (\forall g\in A)\textrm{ and }\theta(X)=-X\}.$$ Note that $(\frg_0^{A},\theta)$ is a
symmetric pair. The above condition is equivalent to: $Z_{(\frg_0^{A})^{-\theta}}(\fra_0)=\fra_0.$
\footnotetext{We might omit the requirement that $\theta$ is an involutive automorphism of $G$, but only
require that $\theta$ is an automorphism of $G$ with $\theta(a)=a^{-1}$ ($\forall a\in A$) and $\fra_0$ is
a maximal abelian subspace of $(\frg_0^{A})^{-\theta}$. Theorem \ref{T:restrictedRS2} still holds with this
generalization.}

\begin{theorem}\label{T:restrictedRS2}
Let $A$ be a symmetric $(\ast)$-subgroup. Then the set $R(G,A^{0})$ of non-zero characters of $A^{0}$ for
the conjugation action of $A^{0}$ on $\frg$ is a root system in the lattice $X^{\ast}(A^{0})$.
\end{theorem}

\begin{proof}
The proof is by combining ideas in the proofs of Theorem \ref{T:restrictedRS1} and \cite[Prop. 6.52]
{Knapp}. Write $R(G,A)$ for the set of non-zero characters of $A$ in $\frg$ for the conjugation action of
$A$ on $\frg$, and write $$R_{0}(G,A)=\{\alpha\in R(G,A):\alpha|_{A^{0}}=1\}.$$ Then, $R(G,A^{0})$ is the
image of $R(G,A)-R_{0}(G,A)$ under the restriction map $X^{\ast}(A)\rightarrow X^{\ast}(A^{0})$. For each
$\alpha\in R(G,A)-R_{0}(G,A)$, define $\check{\alpha}\in X_{\ast}(A^{0})\otimes_{\mathbb{Z}}\mathbb{R}$ by
$$\lambda(\check{\alpha})=\frac{2(\lambda,\alpha|_{A^{0}})}{(\alpha|_{A^{0}},\alpha|_{A^{0}})},\ \forall
\lambda\in X^{\ast}(A^{0}).$$ Set $s_{\alpha}\in\fra_0\rightarrow\fra_0$ by $$s_{\alpha}(X)=X-\alpha(X)
\check{\alpha},\ \forall X\in\fra_0.$$

Analogous to the proof of \cite[Proposition 6.52]{Knapp}, we can construct an $\mathfrak{sl}_{2}$-triple
$\{H_{\alpha},E_{\alpha},\theta E_{\alpha}\}$ such that $H_{\alpha}\in i\fra_0$, $E_{\alpha}\in\frg_{\alpha}$,
and $\theta E_{\alpha}\in\frg_{-\alpha}$. Normalize $E_{\pm{\alpha}}$ appropriately, we can make $$k_{\alpha}
=\exp(\frac{\pi}{2}(E_{\alpha}+\theta E_{\alpha}))\in N_{(G^{0})^{\theta}}(A)$$ and $\Ad(k_{\alpha})|_{A}
=s_{\alpha}$. By this, $s_{\alpha}$ integrates to a reflection $s_{\alpha}: A\rightarrow A$, and it
preserves $R(G,A)$. This implies that $R(G,A^{0})$ is a root system in the lattice $X^{\ast}(A^{0})$.

Alternatively, let $$G_{[\alpha]}=(Z_{G}(\ker\alpha)^{0})_{der}.$$ Then, $G_{[\alpha]}$ is connected
semisimple subgroup of $G$, stable under the conjugation of $A$ and the action of $\theta$, and
$\ker\alpha$ commutes with $G_{[\alpha]}$. Set $T_{\alpha}=(A\cap G_{[\alpha]})^{0}$. Then, $T_{\alpha}$
is one-dimensional torus and $A=\ker\alpha\cdot T_{\alpha}$. Moreover, $\Lie T_{\alpha}$ is a maximal
abelian subspace of $(\frg_{[\alpha]})_{0}^{-\theta}$. Hence, $((\frg_{[\alpha]})_{0},\theta)$ is a
symmetric pair of rank one. By referring to the proof of \cite[Proposition 6.52]{Knapp}, or using the
classification of symmetric pairs of rank one, one constructs $k_{\alpha}\in(G_{[\alpha]})^{\theta}$
such that $\Ad(k_{\alpha})|_{\fra_0}=s_{\alpha}$. By this, $s_{\alpha}$ integrates to a reflection
$s_{\alpha}:A\rightarrow A$, and it preserves $R(G,A)$. It follows that $R(G,A^{0})$ is a root
system in the lattice $X^{\ast}(A^{0})$.
\end{proof}


Note that when $A=\{1\}$, the root system $R(G,A^{0})$ is the restricted root system for the compact
symmetric pair $(\frg_0,\theta)$ (cf. \cite[Corollary 6.53]{Knapp}). When there is an element $g\in G$
such that $\theta g\theta^{-1}=g^{-1}$ and $A=\overline{\langle g\rangle}$, $R(G,A^{0})$ is the restricted
root system associated to two involutions $\theta,g\theta$ as defined in \cite{Oshima-Sekiguchi}.

\smallskip

Theorems \ref{T:restrictedRS1} and \ref{T:restrictedRS2} give a vast examples of tori $B$ in a compact Lie
group $G$ such that the set $R(G,B)$ of non-zero characters of $B$ for the conjugation action of $B$ on
$\frg$ is a root system in the lattice $X^{\ast}(B)$. This is equivalent to: $R(G,B)=\Psi'_{B}$; also
equivalent to: for any $\alpha\in R(G,B)$, there exists a rank one connected closed subgroup $K$ of $G$
with $B$ a maximal torus of $K$ and with $\Delta(K,B)=\{\pm{\alpha}\}$.

\begin{question}\label{Q:RS}
Let $B$ be a torus in a compact Lie group $G$. Shall one give a general criterion for the set $R(G,B)$ of
non-zero characters of $B$ for the conjugation action of $B$ on $\frg$ is a root system in the lattice
$X^{\ast}(B)$?
\end{question}

\subsection{Abstract twisted root system and $(\ast)$-subgroups}\label{SS:abstractTRS}

We give a definition of abstract twisted root system (called TRS for short) as follows, which is subject
to be refined in future.

\begin{definition}\label{D:abstractTRS}
Let $A$ be a compact abelian Lie group with an inner product on its character group, denoted by $(\cdot,\cdot)$.
A twisted root system on $A$ is a pair $(R,R^{\vee})$ subject to the following requirements,
\begin{itemize}
\item[(1)]{$R\subset X^{\ast}(A)$. An element $\alpha\in R-X^{\ast}_{0}(A)$ is called an infinite root;
an element $\alpha\in R\cap X^{\ast}_{0}(A)$ is called a finite root.}
\item[(2)]{(Strong integrality) For an infinite root $\alpha$,
$$\frac{2(\lambda,\alpha)}{(\alpha,\alpha)}\in\mathbb{Z},\ \forall\lambda\in X^{\ast}(A).$$ }
\item[(3)]{For a finite root $\alpha$ of order $m$, $R^{\vee}(\alpha)$ is a subgroup of $X_{\ast1}(\ker\alpha)(n)$
and is called the coroot group of $\alpha$.}
\item[(4)]{If $\alpha$ is an infinite root, define $\check{\alpha}\in X_{\ast}(A)$ by $$\lambda(\check{\alpha})
=\frac{2(\lambda,\alpha)}{(\alpha,\alpha)},\ \forall\lambda\in X^{\ast}(A).$$ Define the reflection
$s_{\alpha}$ by $$s_{\alpha}(a)=a\check{\alpha}(\alpha(a)^{-1}),\ \forall a\in A.$$ Then, the induced action of
$s_{\alpha}$ on characters and finite cocharacters preserves roots and coroot groups. That means $s_{\alpha}(R)=R$,
and $$s(\alpha)(R^{\vee}(\beta))=R^{\vee}(s(\alpha)(\beta))$$ for any finite root $\beta$..}
\item[(5)]{If $\alpha$ is a finite root of order $n$, and $k$ is relatively prime to $n$, then $k\alpha$ is also
a finite root and $R^{\vee}(k\alpha)=R^{\vee}(\alpha)$.}
\item[(6)]{If $\alpha$ is a finite root and $\xi\in R^{\vee}(\alpha)$, define $s_{\alpha,\xi}$ by
$$s_{\alpha,\xi}(a)=a\xi(\alpha(a)),\ \forall a\in A.$$ Then, $s_{\alpha,\xi}$ is an automorphism of $A$ and its
induced action on characters, infinite cocharacters and finite cocharacters induces an automorphism of $\Psi$.
That means $s_{\alpha,\xi}(R)=R$, and $$s(\alpha,\xi)(R^{\vee}(\beta))=R^{\vee}(s(\alpha,\xi)(\beta))$$ for any
finite root $\beta$.}
\end{itemize}
\end{definition}

When $A$ is a finite group, Definition \ref{D:abstractTRS} is the definition of finite root datum given in
\cite{Han-Vogan}. A big difficulty in studying abstract TRS is due to lacking sufficient understanding of
FRD. We have seen that FRD does not recover the Lie algebra (cf. Theorem \ref{T:pdecom}).

Using root-coroot duality, we could also define a notion of twisted root datum by generalizing that in
\cite{Springer-reductive} with incorporating finite roots and their coroot groups. We do not bother ourselves
to give the details here.

Note that we do not require that $R$ is reduced, i.e., it is not required that $2\alpha\not\in R$ if $\alpha
\in R-X^{\ast}_{0}(A)$. We could also define sub-TRS. Let $G$ be a compact Lie group and $A$ be a compact
abelian subgroup, like in \cite[Definition 3.1]{Yu-dimension}, we could associate with $A$ a TRS $\Psi_{A}$
and a sub-TRS $\Psi'_{A}$ of it.

Some arguments used in the main body of this paper enable us to show many properties of abstract twisted root
system. Some of these include: the set $R'=\{\alpha|_{A^{0}}:\alpha\in R-R_{0}\}$ is a root system in the
lattice $X^{\ast}(A^{0})$; shape of the set $\span_{\mathbb{Z}}\{\alpha,\lambda\}$ for $\alpha\in R-R_{0}$
and $\lambda\in X_{0}^{\ast}(A)$; constancy of length of strips when $R'$ is a simply-laced irreducible
root system of rank $\geq 2$.

Our understanding to abstract TRS is still very limited. In the future we wish to refine the axioms of abstract
TRS by further investigating the structure of TRS from a $(\ast)$-subgroup of a compact Lie group. A particular
question is to show the TRS $(R,R^{\vee})$ is very restrictive if we assume it is irreducible (omit the precise
definition now) and $R'$ is an exceptional irreducible root system. The optimal goal is to have a classification
of abstract TRS and to show that they all come from $(\ast)$-subgroups of compact Lie groups.

\end{document}